\documentclass[12pt,reqno]{amsart}

\usepackage[all]{xy}
\usepackage{url}
\usepackage{amssymb}
\usepackage[backref]{hyperref}
\usepackage[margin=1.2 in]{geometry}
\usepackage{color,soul}
\usepackage{amsthm}
\usepackage{times}
\usepackage{mathtools}

\DeclareMathOperator{\Spec}{Spec}

\DeclareMathOperator{\red}{red}
\DeclareMathOperator{\Frac}{Frac}

\DeclareMathOperator{\coker}{coker}
\DeclareMathOperator{\Hom}{Hom}
\DeclareMathOperator{\Aut}{Aut}
\DeclareMathOperator{\ev}{ev}
\DeclareMathOperator{\Mod}{Mod}
\DeclareMathOperator{\Tor}{Tor}

\theoremstyle{plain}
\newtheorem{theorem}[subsubsection]{Theorem}
\newtheorem{proposition}[subsubsection]{Proposition}
\newtheorem{lemma}[subsubsection]{Lemma}

\newtheorem{corollary}[subsubsection]{Corollary}

\theoremstyle{definition}
\newtheorem{definition}[subsubsection]{Definition}
\newtheorem{example}[subsubsection]{Example}

\newtheorem{remark}[subsubsection]{Remark}
\newtheorem{remarks}[subsubsection]{Remarks}

\theoremstyle{definition}

\newcommand{\fM}{\mathfrak{M}}

\newcommand{\cP}{\mathcal{P}}
\newcommand{\fP}{\mathfrak{P}}

\newcommand{\fQ}{\mathfrak{Q}}

\newcommand{\fm}{\mathfrak{m}}

\newcommand{\fn}{\mathfrak{n}}

\newcommand{\fp}{\mathfrak{p}}

\newcommand{\fq}{\mathfrak{q}}

\theoremstyle{plain}
\newtheorem{innercustomthm}{Theorem}
\newenvironment{customthm}[1]
  {\renewcommand\theinnercustomthm{#1}\innercustomthm}
  {\endinnercustomthm}
  
\theoremstyle{plain}
\newtheorem{innercustomprop}{Proposition}
\newenvironment{customprop}[1]
  {\renewcommand\theinnercustomprop{#1}\innercustomprop}
  {\endinnercustomprop}

\theoremstyle{plain}

\numberwithin{equation}{subsubsection}
\numberwithin{equation}{subsubsection}
\normalfont
\usepackage[T1]{fontenc}{}
\usepackage{setspace}

\usepackage{tikz-cd}

\makeatletter
\def\@tocline#1#2#3#4#5#6#7{\relax
  \ifnum #1>\c@tocdepth 
  \else
    \par \addpenalty\@secpenalty\addvspace{#2}%
    \begingroup \hyphenpenalty\@M
    \@ifempty{#4}{%
      \@tempdima\csname r@tocindent\number#1\endcsname\relax
    }{%
      \@tempdima#4\relax
    }%
    \parindent\z@ \leftskip#3\relax \advance\leftskip\@tempdima\relax
    \rightskip\@pnumwidth plus4em \parfillskip-\@pnumwidth
    #5\leavevmode\hskip-\@tempdima
      \ifcase #1
       \or\or \hskip 1em \or \hskip 2em \else \hskip 3em \fi%
      #6\nobreak\relax
    \dotfill\hbox to\@pnumwidth{\@tocpagenum{#7}}\par
    \nobreak
    \endgroup
  \fi}
\makeatother
\usepackage{hyperref}

\title{On some permanence properties of (derived) splinters}
\author{Rankeya Datta}
\address{Department of Mathematics, Statistics and Computer Science\\University
of Illinois at Chicago\\Chicago, IL 60607-7045, USA}
\email{\href{mailto:rankeya@uic.edu}{rankeya@uic.edu}}
\urladdr{\url{https://rankeya.people.uic.edu/}}

\author{Kevin Tucker}
\email{\href{mailto:kftucker@uic.edu}{kftucker@uic.edu}}
\urladdr{\url{https://kftucker.people.uic.edu/}}
\thanks{The second author was supported in part by NSF grant DMS \#1602070 and \#1707661, and by a fellowship from the Sloan Foundation.}

\begin{document}

\maketitle

\begin{abstract}
We show that Noetherian splinters ascend under essentially \'etale homomorphisms. Along the way, we also prove that the Henselization of a Noetherian local splinter is always a splinter and that the completion of a local splinter with geometrically regular formal fibers is a splinter. Finally, we give an example of a (non-excellent) Gorenstein local splinter with mild singularities whose completion is not a splinter. Our results provide evidence for a strengthening of the direct summand theorem, namely that regular maps preserve the splinter property.
\end{abstract}



\section{Introduction}
Recall that a {Noetherian} ring $R$ is a \emph{splinter} if any finite ring map $R \rightarrow S$ that induces a surjection on $\Spec$ has a left-inverse in the category of $R$-modules \cite{Ma88}. Perhaps owing to their simple definition, basic questions about splinters are often devilishly difficult to answer. For example, Hochster's direct summand conjecture (now a theorem) is the modest assertion that a regular ring of any characteristic is a splinter. However, it took the advent of perfectoid geometry for this conjecture to be settled by Andr\'e in mixed characteristic \cite{And18(a)} (see also \cite{Bha18, Hei02}), decades after Hochster's verification of the equal characteristic case using Frobenius techniques \cite{Hoc73(a)}. 

The direct summand theorem justifies thinking of a splinter as a characteristic independent notion of singularity. The goal of this paper is to show that splinters satisfy some basic permanence properties enjoyed by other classes of singularities. Our first main result follows below.


\begin{customthm}{A}
\label{Theorem-A}
Let $\varphi\colon R \rightarrow S$ be an essentially \'etale homomorphism of Noetherian rings of arbitrary characteristic. If $R$ is a splinter, then $S$ is a splinter.
\end{customthm}

\noindent 
Most notions of singularities such as reduced, normal, Gorenstein, complete intersection, Cohen--Macaulay in fact ascend under \emph{regular maps}, that is, flat maps of Noetherian rings with geometrically regular fibers. Theorem \ref{Theorem-A} provides characteristic independent evidence suggesting the same is true for splinters. However, a proof of this stronger result is likely difficult as ascent of splinters under regular maps would imply the direct summand theorem (see Remark \ref{rem:question-1-implies-dsc}(2)). On the other hand, Theorem \ref{Theorem-A} and an application of N\'eron--Popescu desingularization show that the principal remaining difficulty lies in verifying that a polynomial ring over a splinter remains a splinter (Remark \ref{rem:question-1-implies-dsc}(1)).

%

Splinters behave very differently depending on the characteristic of the ring. In equal characteristic $0$, for instance, splinters starkly contrast with a more restrictive derived variant. Following \cite{Bha12}, we say a Noetherian ring $R$ of arbitrary characteristic is a \emph{derived splinter} if for any proper surjective morphism $f: X \rightarrow \Spec(R)$, the induced map $R \rightarrow \textbf{R}\Gamma(X, \mathcal{O}_X)$ splits in the derived category $D(\Mod_R)$. In equal characteristic $0$, splinters have long been known to coincide with normal rings via an argument involving the trace map \cite[Lemma 2]{Hoc73(a)}. On the other hand, a ring essentially of finite type over a field of characteristic $0$ is a derived splinter precisely when it has rational singularities \cite{Kov00} (cf. \cite[Theorem 2.12]{Bha12}). Hence such rings are not only normal but even Cohen--Macaulay.

%

The situation is remarkably different in prime characteristic, where Bhatt showed that splinters coincide with derived splinters \cite[Theorem 1.4]{Bha12}. Furthermore, in this setting splinters are conjecturally equivalent to $F$-regular singularities. $F$-regularity was introduced by Hochster and Huneke in the celebrated theory of tight closure \cite{HH90}, and correspond via standard reduction techniques to the Kawamata log terminal singularities fundamental in complex birational geometry \cite{Smi97, MS97, Har98}. While an $F$-regular singularity is always a splinter (Remark \ref{rem:charp-splinters}(2)), the converse is known in the $\mathbb{Q}$-Gorenstein setting \cite{Sin99} and for prime characteristic rings with finite Frobenius whose anti-canonical ring is finitely generated \cite{CEMS18} (the latter also follows by unpublished work of Singh). The conjectural equivalence of splinters and $F$-regularity would also resolve some important and long--standing localization questions in tight closure theory.  

It is perhaps not surprising that (derived) splinters remain quite mysterious in mixed characteristic, where sophisticated techniques were required to prove the direct summand theorem. Bhatt improved upon Andr\'e's result in \cite{Bha18}, showing that regular rings in mixed characteristic are even derived splinters. Astonishingly, Bhatt's forthcoming work further indicates that splinters and derived splinters also coincide in mixed characteristic \cite{Bha}.

Given the close relationship between splinters and their derived variant, it is  natural to wonder about the the derived analogue of Theorem \ref{Theorem-A}. In equal characteristic zero, this follows from \cite{Kov00} because rational singularities are preserved by essentially \'etale maps (see Corollary \ref{cor:derived-Theorem-A}). In prime and mixed characteristic, Theorem \ref{Theorem-A} also implies the derived analogue by Bhatt's work.

%
%
%
%
%
%
%
%
%
%
%
%

\subsection{Structure of the proof of Theorem \ref{Theorem-A}} We prove Theorem \ref{Theorem-A} by reducing to the case where $\varphi: R \rightarrow S$ is \emph{finite \'etale}, which first necessitates an understanding of how splinters behave under Henselization. To that end, we establish the following result:

\begin{customthm}{B}
\label{Theorem-B}
Let $(R, \fm)$ be a Noetherian local ring. Then $R$ is a splinter if and only if its Henselization $R^h$ is a splinter.
\end{customthm}

\noindent Strict Henselizations of splinters are also splinters (see Corollary \ref{cor:splinter-strict-Henselization}), which in turn implies that Theorems \ref{Theorem-A} and \ref{Theorem-B} are equivalent (Remark \ref{rem:main-theorem-A-B-equivalent}).

Preservation of the splinter property under Henselization raises the natural question of whether the completion of a local splinter is a splinter. In equal characteristic zero this fails by Heitmann's examples of Noetherian local UFDs whose completions are not reduced \cite{Hei93} because splinters are always reduced (Lemma \ref{lem:local-splinters-domain}). Similar examples do not work in prime characteristic where local splinters are always analytically unramified (Remark \ref{rem:charp-splinters}(1)). Nevertheless, we show that there exist positive characteristic local splinters with even Gorenstein $F$-regular singularities whose completions are not splinters (Example \ref{ex:completion-not-splinter}). At the same time, any such example cannot be excellent because we also prove that splinters behave well under completions for local rings that usually arise in arithmetic and geometry. 

\begin{customthm}{C}
\label{Theorem-C}
Let $(R, \fm)$ be a Noetherian local ring such that $R \rightarrow \widehat{R}$ is regular (for example, if $R$ is excellent). Then $R$ is a splinter if and only if $\widehat{R}$ is a splinter.
\end{customthm}

\noindent Theorem \ref{Theorem-C} further corroborates our belief that regular maps should preserve the splinter property.

The proof of Theorem \ref{Theorem-C} uses an ideal-theoretic result of Smith for excellent normal local rings \cite[Proposition 5.10]{Smi94}. In our setting, we are able to establish an analogue of Smith's result for Henselizations of \emph{arbitrary} normal local rings, not just excellent ones. This key result, highlighted below, allows us to prove Theorem \ref{Theorem-B} without any restrictions on formal fibers.

\begin{customprop}{\ref{prop:ideal-identity-hensel}}
Let $(R,\fm)$ be a normal local domain (not necessarily Noetherian), and let $R^h$ and $R^{sh}$ denote its Henselization and strict Henselization respectively. Then $R^h$ and $R^{sh}$ are normal domains, and if $I$ is an ideal of $R$, we have 
$$I(R^h)^+ \cap R = IR^+ \cap R = I(R^{sh})^+ \cap R.$$
\end{customprop}

\noindent Here $R^+$ denotes the \emph{absolute integral closure} of a domain $R$, that is, $R^+$ is the integral closure of $R$ in an algebraic closure of its fraction field. Armed with Proposition \ref{prop:ideal-identity-hensel} and \cite[Proposition 5.10]{Smi94}, Theorems \ref{Theorem-B} and \ref{Theorem-C} then have almost identical proofs.

\subsection{Outline of the paper} We begin Section 2 by discussing some elementary properties of splinters (Subsection \ref{subsec:Basic properties}). We next investigate descent of splinters under pure and, the closely related notion of, cyclically pure maps. Crucial to our comparison of purity versus cyclic purity for splinters is Hochster's notion of approximately Gorenstein rings (Definition \ref{def:approx-Gorenstein}). These are Noetherian rings for which purity and cyclic purity coincide. Splinters are also related to absolute integral closures, a connection hinted at by the aforementioned Proposition \ref{prop:ideal-identity-hensel}. Indeed, it is well-known that in order for a Noetherian domain $R$ to be a splinter, it is necessary and sufficient for $R \rightarrow R^+$ to be a pure map (a.k.a. universally injective map). Thus, we briefly discuss how the absolute integral closure interacts with the notion of a splinter. We end Section \ref{sec:Properties of splinters} by highlighting connections between the splinter condition and notions of singularities defined in prime characteristic via the Frobenius map. In Section \ref{sec:Proof of the main theorems} we prove our main theorems. We end our paper by collecting some basic questions, which, as far as we know, are still open for splinters (Section \ref{sec:open-questions}).

\subsection{Conventions}
All rings in this paper are commutative with identity. Although our results are primarily about Noetherian rings, we often use the absolute integral closure of a domain which is highly non-Noetherian. Diverging from usual practice in commutative algebra, by a `local ring' we mean a ring with a unique maximal ideal which is not necessarily Noetherian. When we want the local ring to be Noetherian, we will explicitly say so. We will sometimes talk about normal rings in a non-Noetherian setting. Recall that an arbitrary commutative ring $R$ is \emph{normal} if for all prime ideals $\fp$ of $R$, $R_\fp$ is a domain which is integrally closed in its fraction field \cite[\href{https://stacks.math.columbia.edu/tag/00GV}{Tag 00GV}]{stacks-project}. A reduced ring $R$ (not necessarily Noetherian) with finitely many minimal primes is normal precisely when it is integrally closed in its total quotient ring, and in this case $R$ decomposes as a finite product of normal domains \cite[\href{https://stacks.math.columbia.edu/tag/030C}{Tag 030C}]{stacks-project}.

\section{Properties of splinters}\label{sec:Properties of splinters}
Throughout this section, $R$ will denote a Noetherian ring of arbitrary characteristic unless otherwise specified. We begin by highlighting some well-known properties of splinters.

\subsection{Basic properties}\label{subsec:Basic properties}
If $R$ is a splinter, then $R$ is reduced because the canonical map $R \twoheadrightarrow R_{\red}$ splits. Note that a reduced ring $R$ is a splinter if and only if every finite extension $R \hookrightarrow S$ splits in the category of $R$-modules. This is because if $R$ is reduced, a finite map $R \rightarrow S$ induces a surjection on $\Spec$ if and only if it is injective.

A local splinter is not just reduced, but even a domain. In the literature this is usually deduced as a consequence of splinters being normal, but we provide a direct elementary proof here.

\begin{lemma}
\label{lem:local-splinters-domain}
A Noetherian local splinter $(R,\fm)$ is a domain. In fact, $R$ is normal.
\end{lemma}

\begin{proof}
Assume $R$ is not a domain. Since $R$ is reduced, it has more than one minimal prime, say $\fp_1, \dots, \fp_n$, where $n \geq 2$. Consider the projection 
\[
\varphi\colon R \rightarrow R/\fp_1 \times \dots \times R/\fp_n.
\]
This is finite and surjective on $\Spec$, and so, it splits as a map of $R$-modules. Let $\phi$ be a left inverse of $\varphi$, and $e_i$ be the standard idempotent of $R/\fp_1 \times \dots \times R/\fp_n$ with a $1$ in the $i$-th spot. Since $e_i$ is annihilated by $\fp_i$ when $R/\fp_1 \times \dots \times R/\fp_n$ is viewed as an $R$-module, it follows that $\phi(e_i) \in \fm$, for all $i$. But then,
\[
1 = \phi(1) = \sum_{i=1}^n \phi(e_i) \in \fm,
\]
which is a contradiction. So $n = 1$, and $R$ is a domain because it is reduced.

We now show $R$ is normal. Let $K = \Frac(R)$, and $a/b \in K$ be integral over $R$. As the module-finite extension
\[
R \hookrightarrow R[a/b]
\]
splits, it must be an isomorphism since $R$ and $R[a/b]$ are torsion-free $R$-modules of the same rank. Thus, $a/b \in R$, proving normality.
\end{proof}

\begin{remark}
\label{reduction-to-domain}
Lemma \ref{lem:local-splinters-domain} implies that a (non-local) splinter decomposes into a finite product of normal domains \cite[\href{https://stacks.math.columbia.edu/tag/030C}{Tag 030C}]{stacks-project}, each of which is easily checked to also be a splinter. Conversely, a finite direct product of splinters is a splinter. Hence most questions about splinters reduce to the domain case.
\end{remark}

The following result shows that the property of being a splinter is local:

\begin{lemma}
\label{lem:splinters-local}
Let $R$ be a Noetherian ring. Then the following are equivalent:
\begin{enumerate}
	\item $R$ is a splinter.
	\item For all prime ideals $\fp \in \Spec(R)$, $R_\fp$ is a splinter.
	\item For all maximal ideals $\fm \in \Spec(R)$, $R_\fm$ is a splinter.
\end{enumerate}
Hence, if $R$ is a splinter, then for any multiplicative set $S \subset R$, $S^{-1}R$ is a splinter.
\end{lemma}

\begin{proof}[Sketch of proof]
For (1) $\Rightarrow$ (2) it suffices to assume $R$ is a domain, and then the result follows by a simple spreading out argument, while (2) $\Rightarrow$ (3) is obvious. Assuming (3), note that it suffices to show that any finite extension $R \hookrightarrow S$ splits. But such a splitting can be checked locally at the maximal ideals where it always holds by the hypothesis of (3).

The final assertion follows from the equivalence of (1)-(3) because localizations of $S^{-1}R$ at its prime ideals coincide with localizations of $R$ at primes that do not intersect $S$. 
\end{proof}

\subsection{Purity and descent of splinters}\label{subsec:Purity and descent of splinters}
This subsection is the technical heart of this paper, as it provides a more tractable criterion for verifying the splinter condition  (Lemma \ref{lem:normal-approx-Gor}). Among applications, we show that splinters descend under a notion that is substantially weaker than faithful flatness (Proposition \ref{prop:splinters-descend-cyc-pur}), which we now introduce.

For any commutative ring $A$, we say that a map of $A$-modules $M \rightarrow N$ is \emph{pure} if for all $A$-modules $P$, the induced map $M \otimes_A P \rightarrow N \otimes_A P$ is injective. Pure maps are sometimes also called \emph{universally injective maps} \cite[\href{https://stacks.math.columbia.edu/tag/058I}{Tag 058I}]{stacks-project}. Closely related to the notion of purity is that of cyclic purity, which may be less familiar to the reader.

\begin{definition}
\label{def:cyclic-purity}
Given a ring $A$, we say that a map of $A$-modules $M \rightarrow N$ is \emph{cyclically pure} if for all cyclic $A$-modules $A/I$, the induced map $M/IM \rightarrow N/IN$ is injective.
\end{definition}

\begin{remark}
\label{rem:cyclic-purity-vs-purity-Frobenius}
Taking the cyclic module to be $A$ itself, it follows that cyclically pure maps are injective. Pure maps are obviously cyclically pure, and it is easy to check that faithfully flat ring maps are pure, hence cyclically pure \cite[Chapter I, $\mathsection 3.5$, Proposition 9]{Bou89}.  Purity and cyclic purity are significantly weaker than faithful flatness. For example, Kunz showed that the Frobenius map of a Noetherian ring of prime characteristic is faithfully flat precisely when the ring is regular \cite{Kun69}. However, non-regular Noetherian rings for which the Frobenius map is pure (equivalently, cyclically pure) are abundant and are at the heart of Frobenius splitting techniques in positive characteristic algebra and geometry. Note
that if $R$ is a Noetherian ring of prime characteristic $p > 0$ for which the Frobenius map
$F : R \rightarrow F_*R$ is cyclically pure, then to show that
$F$ is pure, it suffices to show by \cite[Theorem 2.6]{Hoc77} that
$R$ is approximately Gorenstein (see Definition \ref{def:approx-Gorenstein}). Since purity and approximately
Gorenstein are local properties, and since cyclic purity of a ring map is preserved under
localization, we may assume that $R$ is a Noetherian local ring
with maximal ideal $\frak m$ whose Frobenius map is cyclically pure. Now to show that $R$ is approximately Gorenstein it suffices
to show that the $\frak m$-adic completion $\widehat{R}$ is 
reduced; see
\cite[Theorem (1.7) and Corollary (2.2)(c)]{Hoc77}. 
For the latter, upon 
identifying $\widehat{R}$ with
\[
\varprojlim_{e \in \mathbf{Z}_{\geq 0}} R/\fm^{[p^e]},
\]
assume some element $\alpha = (r_e + m^{[p^e]})_e \in \widehat{R}$ is
nilpotent. We want to show $\alpha = 0$. Choose $f > 0$ such that $\alpha^{p^f} = 0$. Then for all $e \geq 0$, 
$r^{p^f}_{f+e} \in \fm^{[p^{f+e}]} = (\fm^{[p^e]})^{[p^f]}$. The $f$-th iterate
of the Frobenius map $F^f: R \rightarrow F^f_*R$ is cyclically pure because compositions of cyclically pure maps are cyclically pure. It then follows that
$r_{f+e} \in \fm^{[p^e]}$ for all $e \geq 0$. Finally, because 
$r_{f+e} \equiv r_e$ mod 
$\fm^{[p^e]}$, we see that for all $e \geq 0$, $r_e \in \fm^{[p^e]}$, that is,
$\alpha = 0$, as desired.

\end{remark}

A split map of modules is pure, and the converse is true under a mild restriction.

\begin{lemma}\cite[Corollary 5.2]{HR76}
\label{lem:pure-implies-split}
If $\varphi: M \rightarrow N$ is a pure map of $A$-modules whose cokernel is finitely presented, then $\varphi$ splits, that is, it admits a left inverse in $\Mod_A$.
\end{lemma}

It is natural to ask when the notions of cyclic purity and purity coincide for ring maps. Hochster discovered a surprising algebraic property that characterizes those Noetherian rings $A$ for which any cyclically pure ring map $A \rightarrow B$ is pure. We now introduce this property.

\begin{definition}\cite[Definitions (1.1) and (1.3)]{Hoc77}
\label{def:approx-Gorenstein}
A Noetherian local ring $(R,\fm)$ is \emph{approximately Gorenstein} if it satisfies the following equivalent conditions:
\begin{enumerate}
	\item[(i)] For every integer $N > 0$, there is an ideal $I \subseteq \fm^N$ such that $R/I$ is Gorenstein.
	\item[(ii)] For every integer $N > 0$, there is an $\fm$-primary irreducible ideal $I \subseteq \fm^N$.
\end{enumerate}
We say a Noetherian ring $R$ (not necessarily local) is \emph{approximately Gorenstein} if $R_\fm$ is an approximately Gorenstein local ring for all maximal ideals $\fm$.
\end{definition} 

\begin{remark}
\label{rem:comment-def-approx-Gor}
{\*}
\begin{enumerate}
\item The key point is that if $R$ is approximately Gorenstein, then a ring homomorphism $R \rightarrow S$ is pure if and only if it is cyclically pure \cite[Theorem 2.6]{Hoc77}.

\item If $(R,\fm)$ is Noetherian local, then for every ideal $I \subseteq \fm^N$ such that $R/I$ is Gorenstein, there exists an $\fm$-primary ideal $J$ such that $I \subseteq J \subseteq \fm^N$ and $R/J$ is Gorenstein. Namely, if $x_1,\dots,x_k \in R$ are such that their images in $R/I$ form a system of parameters of $R/I$, then one can take $J = I + (x^N_1,\dots,x^N_k)$. Note that $J$ is irreducible because the zero ideal of a zero-dimensional Gorenstein ring is irreducible.
\end{enumerate}
\end{remark}

The next lemma is a crucial in our proofs of Theorems \ref{Theorem-A}, \ref{Theorem-B} and \ref{Theorem-C}, and establishes a connection between approximately Gorenstein rings and splinters. Although the result is essentially contained in \cite{Hoc77}, we include a proof for the reader's convenience. 

\begin{lemma}\cite{Hoc77}
\label{lem:normal-approx-Gor}
Let $(R,\fm)$ be a Noetherian normal local ring. Then for any ring map $R \rightarrow S$, the following are equivalent:
\begin{enumerate}
	\item $R \rightarrow S$ is pure.
	\item $R \rightarrow S$ is cyclically pure.
	\item For all $\fm$-primary ideals $I$ of $R$, the induced map $R/I \rightarrow S/IS$ is injective.
	\item There exists a decreasing sequence $\{I_t\}_{t \in \mathbb N}$ of $\fm$-primary ideals of $R$ cofinal with powers of $\fm$ such that for all $t$, $R/I_t$ is a Gorenstein ring, and the induced map $R/I_t \rightarrow S/I_tS$ is injective.
\end{enumerate}
In particular, if $R$ is a splinter, then for any map $R \rightarrow S$ conditions $(1)$-$(4)$ are equivalent.
\end{lemma}

\begin{proof}
The implications $(1) \Rightarrow (2) \Rightarrow (3)$ follow by the definitions of purity and cyclic purity. For $(3) \Rightarrow (4)$ we need to construct a decreasing sequence $\{I_t\}$ of $\fm$-primary ideals cofinal with powers of $\fm$ such that each $R/I_t$ is Gorenstein. For this, it suffices to show by Remark \ref{rem:comment-def-approx-Gor} that a Noetherian normal local ring is approximately Gorenstein. We may assume that $\dim(R) \geq 2$. Otherwise $R$ is a regular local ring, and regular local rings are clearly approximately Gorenstein. 
If $\dim(R) \geq 2$ and $R$ is normal, then $R$ has depth $\geq 2$. Hence the depth of the completion $\widehat{R}$ is also $\geq 2$, and so, $R$ is approximately Gorenstein by \cite[Theorem (5.2)]{Hoc77} because $\widehat{R}$ has no associated primes of coheight $\leq 1$.

It remains to show $(4) \Rightarrow (1)$. Let $E = E_R(R/\fm)$ be the injective hull of the residue field of $R$. Recall that by Matlis duality, $R \rightarrow S$ is pure if and only if the induced map $E \rightarrow E \otimes_R S$ is injective \cite[Lemma 2.1(e)]{HH95}. Since every element of $E$ is annihilated by a power of $\fm$, it follows by the hypotheses on the collection $\{I_t\}$ that
\[
E = \bigcup_t  (0:_E I_t).
\]
Now $(0:_E I_t) \cong \Hom_R(R/I_t, E)$ is an injective $R/I_t$-module by $\Hom$--$\otimes$ adjunction. 
One then checks that $(0:_E I_t)$ is in fact the injective hull of the residue field of $R/I_t$, and hence coincides with the latter zero-dimensional Gorenstein local ring \cite[Theorem A.29]{twentyfour}. Thus $E$ is the union of $R$-modules isomorphic to $R/I_t$, and injectivity of $E \rightarrow E \otimes_R S$ then follows because for every $t$, $R/I_t \rightarrow S/I_tS$ is injective by the hypothesis of (4).
\end{proof}

\begin{remark}
Lemma \ref{lem:normal-approx-Gor} is only stated for a Noetherian normal ring, although its proof holds for any approximately Gorenstein ring. Apart from Noetherian normal rings, the class of approximately Gorenstein rings includes Noetherian local rings that are formally reduced \cite[Theorem (1.7) and Corollary (2.2)(c)]{Hoc77}, excellent reduced rings \cite[Theorem (1.7)]{Hoc77}, and Noetherian local rings of depth at least $2$ \cite[Theorem (1.6)]{Hoc77} because such rings have no
associated primes of coheight $1$ by \cite[\href{https://stacks.math.columbia.edu/tag/0BK4}{Tag 0BK4}]{stacks-project}.
\end{remark}

Lemma \ref{lem:normal-approx-Gor}  can be used to show that splinters descend under cyclically pure maps.

\begin{proposition}
\label{prop:splinters-descend-cyc-pur}
Let $\varphi\colon R \rightarrow S$ be a map of Noetherian rings such that $S$ is splinter.
\begin{enumerate}
	\item If $\varphi$ is pure, then $R$ is a splinter. In particular, splinters descend under faithfully flat maps.	
	\item If $\varphi$ is cyclically pure and maps nonzerodivisors of $R$ to nonzerdivisors of $S$ (for example, if $S$ is a domain or $\varphi$ is flat), then $R$ is a splinter.
\end{enumerate}
\end{proposition}

\begin{proof}
Note that in both cases $R$ is reduced since $S$ is reduced and $\varphi$ is injective. Assertion (2) follows from (1). Indeed, we claim that the hypotheses of (2) imply that $\varphi$ is pure. This follows by Lemma \ref{lem:normal-approx-Gor} if we can show that $R$ is normal. Since $R$ is Noetherian, it suffices for us to show that it is integrally closed in its total quotient ring \cite[\href{https://stacks.math.columbia.edu/tag/030C}{Tag 030C}]{stacks-project}. Note that $S$ is normal since $S$ is a splinter (Lemma \ref{lem:local-splinters-domain}). Let $a/b$ be an element in the total quotient ring of $R$ (hence $b$ is a nonzerodivisor) that is integral over $R$. Then $\varphi(a)/\varphi(b)$ is an element in the total quotient ring of $S$ that is integral over $S$ ($\varphi(b)$ remains a nonzerodivisor in $S$). Therefore $aS \subseteq bS$ since $S$ is normal, and so, $aR = aS \cap R \subseteq bS \cap R = bR$ by cyclic purity of $\varphi$. This shows $a/b \in R$, and so, $R$ is normal. 

For the proof of (1), let $R \rightarrow T$ be a finite map such that the induced map $\Spec(T) \rightarrow \Spec(R)$ is surjective. By base change, $S \rightarrow T \otimes_R S$ has the same properties (a surjective morphism of schemes is `universally surjective' by \cite[Chapter 1, Proposition (3.5.2)(ii)]{EGAI}). Therefore, the map
\[
S \rightarrow T \otimes_R S
\]
splits in the category of $S$-modules. In particular, $S \rightarrow T \otimes_R S$ is a pure map of $S$-modules (hence also $R$-modules). Consider the commutative diagram
\[
\begin{tikzcd}
  R \arrow[r, "\varphi"] \arrow[d]
    & S \arrow[d] \\
  T \arrow[r]
&T \otimes_R S 
\end{tikzcd}
\]
Since the composition $R \xrightarrow{\varphi} S \rightarrow S \otimes_R T$ is pure as a map of $R$-modules, it follows that $R \rightarrow T$ is also pure. But $\coker(R \rightarrow T)$ is a finitely presented $R$-module since $R$ is Noetherian and $R \rightarrow T$ is finite, and so, $R \rightarrow T$ splits by Lemma \ref{lem:pure-implies-split}.
\end{proof}

\subsection{Absolute integral closure and splinters}
Recall that if $R$ is a domain, then its \emph{absolute integral closure}, denoted $R^+$, is the integral closure of $R$ in a fixed algebraic closure of its fraction field. Absolute integral closures allow one to verify the splinter condition by checking cyclic purity of a \emph{single} map, as highlighted in the following result.

\begin{lemma}
\label{lem:absolute-int-cl-splinter}
Let $R$ be a Noetherian domain. Then $R$ is a splinter if and only if the map $R \rightarrow R^+$ is cyclically pure.
\end{lemma}

\begin{proof}
Suppose $R$ is a splinter. The map $R \rightarrow R^+$ is a filtered colimit of extensions of form $R \hookrightarrow T$, where $T$ is a finitely generated $R$-subalgebra of $R^+$. But $T$ is then a finite $R$-module (since it is integral over $R$), and so, $R \hookrightarrow T$ splits. Since a filtered colimit of split maps is pure \cite[Lemma 2.1(i)]{HH95}, we see that $R \rightarrow R^+$ is cyclically pure.

Conversely, suppose $R \rightarrow R^+$ is cyclically pure. Because $R^+$ is integrally closed in its fraction field, the same reasoning as in the proof of Proposition \ref{prop:splinters-descend-cyc-pur}(2) implies that $R$ is normal. Thus, $R \rightarrow R^+$ is pure by Lemma \ref{lem:normal-approx-Gor}. Let $R \rightarrow S$ be a module finite ring map that is surjective on $\Spec$, and $\fp \in \Spec(S)$ be a prime ideal that contracts to $(0) \in \Spec(R)$. Then $R \hookrightarrow S/\fp$ is module finite extension of domains, and hence, $S/\fp$ embeds in $R^+$. As the composition $R \hookrightarrow S/\fp \hookrightarrow R^+$ is pure, so is $R \hookrightarrow S/\fp$. But then $R \hookrightarrow S/\fp$ splits by Lemma \ref{lem:pure-implies-split}, which shows that $R \rightarrow S$ also splits.
\end{proof}

\subsection{Connection with $F$-singularities}\label{subsec:splinters-F-singularities}
Splinters have long been known to be related to singularities in prime characteristic defined via the Frobenius map. This connection will be important in our analysis of how splinters behave under completions (see Example \ref{ex:completion-not-splinter}). Hence we briefly recall some relevant definitions.

If $I$ is an ideal of a Noetherian domain $R$ of prime characteristic $p > 0$, then the \emph{tight closure of $I$}, denoted $I^*$, is the collection of elements $r \in R$ for which there exists a nonzero $c \in R$ such that $cr^{p^e} \in I^{[p^e]}$, for all $e \gg 0$. Here $I^{[p^e]}$ denotes the ideal generated by $p^e$-th powers of elements of $I$. We say that $R$ is \emph{weakly $F$-regular} if all ideals of $R$ are tightly closed, that is, $I^* = I$, for any ideal $I$ of $R$. We say an ideal $I$ of $R$ is a \emph{parameter ideal} if $I$ is generated by elements $r_1,\dots,r_n \in I$ such that for any prime ideal $\fp$ of $R$ that contains $I$, the images of $r_i$ in $R_\fp$ form part of a system of parameters of $R_{\fp}$. 
We say $R$ is \emph{$F$-rational} if every parameter ideal of $R$ is tightly closed. Thus, weakly $F$-regular rings are $F$-rational, and the converse holds when $R$ is Gorenstein \cite[Corollary 4.7]{HH94(a)}.

The connections between splinters and $F$-singularities are summarized below.

\begin{remark}
\label{rem:charp-splinters}
{\*}
\begin{enumerate}
	\item If $R$ is a splinter of prime characteristic, then the Frobenius map is pure, that is, $R$ is \emph{$F$-pure}. More generally, any integral map $R \rightarrow S$ which is surjective on $\Spec$ is pure. If $(R,\fm)$ is in addition local, then purity of Frobenius implies that the completion $\widehat{R}$ is reduced, that is, $R$ is formally reduced. For this last assertion, see 
the argument given in Remark \ref{rem:cyclic-purity-vs-purity-Frobenius}.

Note that local splinters in equal characteristic $0$ need not be formally reduced because there exist Noetherian normal local rings which are not formally reduced. Indeed, Heitmann shows in
\cite[Theorem 8]{Hei93} that any equal characteristic complete local ring
of depth $\geq 2$ can always be realized as a completion of a
Noetherian local UFD. Thus, for example, the non-reduced complete ring 
$\mathbb{C}[[x,y,z]]/(x^2)$ will be the completion of a local splinter.
		
	\item If $R$ is a weakly $F$-regular domain, then $R$ is a splinter. Indeed, for any ideal $I$ of $R$, $IR^+ \cap R = I$ because 
\begin{equation}
\label{eq:tight-forcing}
IR^+ \cap R \subseteq I^* = I,
\end{equation}
where the containment in (\ref{eq:tight-forcing}) follows by \cite[Corollary 5.23]{HH94(b)}. This shows that $R \rightarrow R^+$ is cyclically pure, and so, $R$ is a splinter by Lemma \ref{lem:absolute-int-cl-splinter}.
	
	\item Lemma \ref{lem:absolute-int-cl-splinter} combined with Hochster and Huneke's famous characterization of $R^+$ being a big Cohen-Macaulay algebra \cite[Theorem 5.15]{HH92} can be used to deduce that a locally excellent\footnote{A Noetherian ring $R$ is \emph{locally excellent} if for all maximal ideals $\fm$ of $R$, $R_\fm$ is excellent.} Noetherian splinter of prime characteristic is always Cohen-Macaulay. Since splinters in equal characteristic $0$ are equivalent to normal rings, it follows that splinters are not always Cohen-Macaulay in general. 
		
	\item Smith showed that if $R$ is a locally excellent Noetherian domain of prime characteristic, then for any parameter ideal $I$ of $R$, $I^* = IR^+ \cap R$ \cite[Theorem 5.1]{Smi94}. Thus, all parameter ideals of a locally excellent splinter domain $R$ of prime characteristic are tightly closed ($IR^+ \cap R = I$ by cyclic purity of $R \rightarrow R^+$), that is, $R$ is $F$-rational. $F$-rationality is the prime characteristic analogue of the characteristic $0$ notion of rational singularities, and the latter are derived splinters \cite{Kov00}. However, $F$-rational singularities need not be (derived) splinters because there exist $F$-rational section rings of $\mathbb{P}^1$ for which the Frobenius map is not pure \cite[Example 4.4]{Wat91}. 
\end{enumerate}
\end{remark}

\section{Proofs of the main theorems}\label{sec:Proof of the main theorems}

The goal of this section is to prove Theorems \ref{Theorem-A}, \ref{Theorem-B} and \ref{Theorem-C}. Since the proof of Theorem \ref{Theorem-A} uses Theorem \ref{Theorem-B}, we prove the latter result first.

\subsection{Henselization of a splinter}\label{subsec:Henselization of splinters}
The Henselization of a local ring is constructed as a filtered colimit of certain essentially \'etale $R$-algebras, so we briefly review the notion of essentially \'etale maps first.

A local homomorphism of local rings (not necessarily Noetherian)
\[
\varphi\colon (R,\fm) \rightarrow (S,\fn)
\]
is an \emph{\'etale homomorphism of local rings} if $\varphi$ is flat, $S$ is the localization of a finitely presented $R$-algebra, $\fm S = \fn$ and the induced map $\kappa(\fm) \hookrightarrow \kappa(\fn)$ is finite separable.  A homorphism $\varphi\colon R \rightarrow S$ of rings (not necessarily local) is \emph{\'etale at $\fq \in \Spec(S)$} if the induced map 
$
R_{\varphi^{-1}(\fq)} \rightarrow S_\fq
$
is an \'etale homomorphism of local rings. We say $\varphi$ is \emph{essentially \'etale} (resp. \emph{\'etale}) if $S$ is the localization of a finitely presented (resp. is a finitely presented) $R$-algebra and $\varphi$ is \'etale at all $\fq \in \Spec(S)$. In particular, essentially \'etale maps are flat. 

\begin{remark}
\label{rem:local-etale-is-essentially-etale}

If $\varphi: (R,\fm) \rightarrow (S,\fn)$ is an \'etale homomorphism of local rings, choose a finitely presented $R$-algebra $S'$ such that $S$ is the localization of $S'$ at a prime ideal $\fq$ of $S'$. Since the locus of primes of $\Spec(S')$ at which $\Spec(S') \rightarrow \Spec(R)$ is \'etale is open, we may assume that $S'$ an \'etale $R$-algebra. The upshot is that any \'etale homomorphism of local rings is the localization of an honest \'etale map which can even be chosen to be standard \'etale \cite[\href{https://stacks.math.columbia.edu/tag/00UE}{Tag 00UE}]{stacks-project}. In other words, an \'etale homomorphism of local rings is essentially \'etale. We will frequently use this observation in the sequel.
\end{remark}


Let $(R,\fm)$ be a local ring (not necessarily Noetherian), and let $R^h$ denote its \emph{Henselization}. Recall that $R^h$ is realized as a filtered colimit of \'etale homomorphisms of local rings $(R, \fm) \hookrightarrow (S, \fn)$ such that $\kappa(\fm) \xrightarrow{\sim} \kappa(\fn)$ \cite[Chapter VIII, Th\'eor\`eme 1]{Ray70}. For a choice of a separable algebraic closure $\kappa(\fm)^{sep}$ of the residue field $\kappa(\fm)$ of $R$, one also has the \emph{strict Henselization} $R^{sh}$ of $R$ which is a filtered colimit of \'etale homomorphisms of local rings $(R, \fm) \hookrightarrow (S, \fn)$ such that $\kappa(\fm) \subseteq\kappa(\fn) \subseteq \kappa(\fm)^{sep}$ \cite[Chapter VIII, Th\'eor\`eme 2]{Ray70}. It follows that if $R$ is reduced, regular, normal, Cohen-Macaulay or Gorenstein, then so are $R^h$ and $R^{sh}$ (see \cite[\href{https://stacks.math.columbia.edu/tag/07QL}{Tag 07QL}]{stacks-project}). Additionally, the associated local maps
\[
R \rightarrow R^h \rightarrow R^{sh}
\]
are always faithfully flat \cite[\href{https://stacks.math.columbia.edu/tag/07QM}{Tag 07QM}]{stacks-project}.

When $(R,\fm)$ is Noetherian local, the completion of $R^h$ at its maximal ideal $\fm R^h$ coincides with $\widehat{R}$, that is, the induced map $\widehat{R} \rightarrow \widehat{R^h}$ on completions is an isomorphism \cite[\href{https://stacks.math.columbia.edu/tag/06LJ}{Tag 06LJ}]{stacks-project}. This implies the following relation between $\fm$-primary ideals of $R$ and $\fm R^h$-primary (resp. $\fm \widehat{R}$-primary) ideals of $R^h$ (resp. $\widehat{R}$).

\begin{lemma}
\label{lem:m-primary-ideals-Henselization}
Let $\varphi: (R, \fm) \rightarrow (S, \fn)$ be a local homomorphism of  Noetherian local rings such that the induced map on completions $\widehat{R} \rightarrow \widehat{S}$ is an isomorphism (for example, if $S = R^h, \widehat{R}$). Then expansion and contraction of ideals via $\varphi$ induces a bijection between $\fm$-primary ideals of $R$ and $\fn$-primary ideals of $S$.
\end{lemma}

\begin{proof}
Our hypothesis implies that the composition $R \xrightarrow{\varphi} S \rightarrow \widehat{S}$ coincides, up to isomorphism, with the completion $R \rightarrow \widehat{R}$ which is faithfully flat. In particular, $\varphi$ is a pure map, which implies that expansion of ideals gives an injective map from the collection of ideals of $R$ to the collection of ideals of $S$. Thus, to prove the lemma, it remains to show that any $\fn$-primary ideal of $S$ is the expansion of an $\fm$-primary ideal of $R$ and that $\fm$-primary ideals of $R$ expand to $\fn$-primary ideals of $S$. 

Now because $(\fm S)\widehat{S} = (\fm \widehat{R})\widehat{S} = \fn\widehat{S}$, faithful flatness of $S \rightarrow \widehat{S}$ shows that $\fm S = \fn$. Tensoring the isomorphism $\widehat{R} \xrightarrow{\sim} \widehat{S}$ by $R/\fm^n$ then implies that for all $n \geq 1$, we have
\[
R/\fm^n \xrightarrow{\sim} \widehat{R}/\fm^n\widehat{R} \xrightarrow{\sim} \widehat{S}/\fm^n\widehat{S} \xrightarrow{\sim} S/\fn^n.
\]
Since an $\fm$-primary ideal of $R$ contains $\fm^n$ for some $n \gg 0$, it follows from the above isomorphisms that $\fm$-primary ideals of $R$ expand to $\fn$-primary ideals of $S$.
Let $I$ be an $\fn$-primary ideal of $S$, and choose $n \geq 1$ such that $\fn^n \subseteq I$. If $\overline{I}$ is the image of $I$ in $S/\fn^n$, then by the isomorphism $R/\fm^n \xrightarrow{\sim} S/\fn^n$, there exist $r_1, \dots, r_m \in R$ whose images in $S/\fn^n$ generate the ideal $\overline{I}$. Then $I$ is the expansion of the $\fm$-primary ideal $\fm^n + (r_1,\dots,r_m)$.
\end{proof}

\begin{remarks}
\label{rem:etale-local-localization-honest-etale}
{\*}
\begin{enumerate}
 \item Lemma \ref{lem:m-primary-ideals-Henselization} does not hold if  $S = R^{sh}$ is the strict Henselization of $R$, since in this case the maps $R/\fm^n \rightarrow R^{sh}/\fm^n R^{sh}$ are no longer isomorphisms. 
 
 \item In the statement of Lemma \ref{lem:m-primary-ideals-Henselization}, if the induced map $\widehat{R} \rightarrow \widehat{S}$ is surjective, then every $\fn$-primary ideal of $S$ is the expansion of an $\fm$-primary ideal of $R$. However, uniqueness is lost.
\end{enumerate} 
\end{remarks}

The goal in the rest of this subsection is to show that the splinter condition is preserved under Henselization (Theorem \ref{Theorem-B}). But first, we establish an ideal theoretic result inspired by \cite[Proposition 5.10]{Smi94}. Smith's result is proved under excellence hypothesis and deals with the completion of Noetherian local rings. Surprisingly, it turns out that an analogue of her result holds for henzelizations of arbitrary non-Noetherian normal local domains. Thus we carefully prove the result we need, which should be of independent interest.

\begin{proposition}[c.f. {\cite[Proposition 5.10]{Smi94}}]
\label{prop:ideal-identity-hensel}
Let $(R,\fm)$ be a normal local domain (not necessarily Noetherian), and let $R^h$ and $R^{sh}$ denote its Henselization and strict Henselization respectively. Then we have the following:
\begin{enumerate}
 \item $R^h$ (resp. $R^{sh}$) is a normal domain.  
 \item If $I$ is an ideal of $R$, then $I(R^h)^+ \cap R = IR^+ \cap R = I(R^{sh})^+ \cap R$.
\end{enumerate}
\end{proposition}

\begin{proof}
$(1)$ The fact that $R^h$ and $R^{sh}$ are normal local domains is a consequence of the ascent of normality under \'etale maps \cite[\href{https://stacks.math.columbia.edu/tag/06DI}{Tag 06DI}]{stacks-project}. 

$(2)$ Since $(R^h)^+$ and $(R^{sh})^+$ are $R^+$-algebras,  $IR^+ \cap R$ is contained in $I(R^h)^+ \cap R$ and  $I(R^{sh})^+ \cap R$. Let $z \in I(R^h)^+ \cap R$ (resp. $z \in I(R^{sh})^+ \cap R$). Then there exists a module finite extension domain $T$ of $R^h$ (resp. of $R^{sh}$) contained in $(R^h)^+$ (resp. $(R^{sh})^+$) such that $z \in IT$. Let $i_1,\dots,i_n \in I$ and $t_1,\dots, t_n \in T$ such that 
\[
z = i_1t_1 + \dots + i_nt_n.
\]
Note that the $t_i$ are integral over $R^h$ (resp. $R^{sh}$). Since $R^h$ (resp. $R^{sh}$) is a filtered colimit of \'etale local $R$-algebras, there exists an \'etale homomorphism of local rings $(R,\fm) \hookrightarrow (S,\fn)$ such that $S \subseteq R^h$ (resp. $S \subseteq R^{sh}$) and $t_1,\dots,t_n$ are integral over $S$. 

By Remark \ref{rem:local-etale-is-essentially-etale}, choose an \'etale $R$-algebra $S'$ such that $S$ is the localization of $S'$ at a prime ideal $\fq$ of $S'$. Since $R$ is a normal domain and $R \rightarrow S'$ is \'etale, $S'$ is a normal ring by \cite[\href{https://stacks.math.columbia.edu/tag/033C}{Tag 033C}]{stacks-project}. Moreover, $S'$ has only finitely many minimal primes because all minimal primes of $S'$ must contract to $(0)$ in $R$ by \emph{going down} \cite[\href{https://stacks.math.columbia.edu/tag/00HS}{Tag 00HS}]{stacks-project}, and the generic fiber of $R \rightarrow S'$ is a finite product of finite separable field extensions of $\Frac(R)$. In particular, $S'$ decomposes as a finite product of normal domains by \cite[\href{https://stacks.math.columbia.edu/tag/030C}{Tag 030C}]{stacks-project}, and since $S$ is the localization of one these factors, we may assume $S'$ is a domain. Thus, $S' \subseteq S \subseteq T$. Furthermore, after localizing $S'$ at a suitable element not in $\fq$, we may also assume $t_1, \dots, t_n$ are integral over $S'$. Let $T'$ be the $S'$-subalgebra $S'[t_1,\dots,t_n]$ of $T$.  Then $T'$ is a module-finite extension domain of $S'$ by construction, and 
\[
z \in IT' \cap R.
\] 
Note that $T'$ is not necessarily an integral extension of $R$ (otherwise we would be done). However, the composition $R \rightarrow S' \rightarrow T'$ is quasi-finite since \'etale maps and finite maps are quasi-finite \cite[\href{https://stacks.math.columbia.edu/tag/00U5}{Tag 00U5} and \href{https://stacks.math.columbia.edu/tag/00PM}{Tag 00PM}]{stacks-project} and the composition of quasi-finite maps is quasi-finite \cite[\href{https://stacks.math.columbia.edu/tag/00PO}{Tag 00PO}]{stacks-project}. Moreover, the induced map $\Spec(T') \rightarrow \Spec(R)$ is surjective because $R \rightarrow S'$ is faithfully flat (it is \'etale, $(R,\fm)$ is local and $\fq$ lies over $\fm$)  and $S' \rightarrow T'$ is a finite extension. Then $z \in IR^+ \cap R$ using the following Lemma which makes no further use of the Henselian property.

\begin{lemma}
\label{lem:quasifin-plus-closure}
Let $R \hookrightarrow S$ be a quasi-finite extension of domains (not necessarily Noetherian) such that the induced map on $\Spec$ is surjective. Suppose also that $R$ is integrally closed in its fraction field, that is, $R$ is normal. Then we have the following: 
\begin{enumerate}
	\item $S$ can be identified as a subring of $\Frac(R^+)$. 
	\item With the identification from $(1)$, if $I$ is an ideal of $R$ and $z \in IS \cap R$, then $z \in IR^+ \cap R$.
	\item With the identification from $(1)$, if $I$ is an ideal of $R$, then $IS^+ \cap R = IR^+ \cap R$. 
	\item If $R$ is a splinter, then $R \rightarrow S$ is pure.
\end{enumerate}
\end{lemma}

Recall that a finite type map $\varphi: R \rightarrow S$ is \emph{quasi-finite at $\fq \in \Spec(S)$}, if $\fq$ is isolated in its fibre. We say $\varphi$ is \emph{quasi-finite} if it is quasi-finite at all $\fq \in \Spec(S)$ \cite[\href{https://stacks.math.columbia.edu/tag/00PL}{Tag 00PL}]{stacks-project}, or equivalently, if $\varphi$ has finite fibers \cite[\href{https://stacks.math.columbia.edu/tag/00PM}{Tag 00PM}]{stacks-project}. Finite maps, unramified maps and \'etale maps are quasi-finite. Quasi-finite maps are `close' to being finite by Zariski's Main Theorem -- if $R \rightarrow S$ is quasi-finite at $\fq \in \Spec(S)$ and $\overline{S}$ is the integral closure of $R$ in $S$, then there exists $g \in \overline{S}$ such that $g \notin \fq$ and $\overline{S}_g \xrightarrow{\sim} S_g$ \cite[\href{https://stacks.math.columbia.edu/tag/00Q9}{Tag 00Q9}]{stacks-project}. Since ${S}_g$ is of finite type over $R$, it is then easy to see that $S_g = T_g$, for an $R$-subalgebra $T \subseteq \overline{S}$ such that $R \rightarrow T$ is finite
\cite[\href{https://stacks.math.columbia.edu/tag/00QB}{Tag 00QB}]{stacks-project}.

\noindent\emph{Proof of Lemma \ref{lem:quasifin-plus-closure}.}
We first prove $(1)$ and $(2)$. Note that once $S$ is identified as a subring of $\Frac(R^+)$, $(2)$ will follow if we can show that for each prime ideal $\fP$ of $R^+$, 
\[
z \in (IR^+)R^+_\fP = IR^+_\fP.
\]
Choose any $\fP \in \Spec(R^+)$, and let $\fp \coloneqq \fP \cap R$. Let $\overline{S}$ denote the integral closure of $R$ in $S$, and $\fq$ be a prime ideal of $S$ such that 
\[
\fq \cap R = \fp.
\]
Such a prime exists because $R \hookrightarrow S$ is surjective on $\Spec$. Then by Zariski's Main Theorem \cite[\href{https://stacks.math.columbia.edu/tag/00Q9}{Tag 00Q9}]{stacks-project}, there exists $g \in \overline{S}$, $g \notin \fq$ and an isomorphism 
\[
\overline{S}_g \xrightarrow{\sim} {S}_g.
\]
This implies $\Frac(S) = \Frac(\overline{S})$, and so, $\Frac(S)$ is an algebraic extension of $\Frac(R)$. In particular, we may identify $\Frac(S)$, hence also $S$ and  $\overline{S}$, as subrings of $\Frac(R^+)$, which proves $(1)$.

With the identification from $(1)$, we have $R \subseteq \overline{S} \subseteq R^+$. Thus, $\overline{S} \subseteq R^+$ is also an integral extension, and consequently,
\[
\overline{S}_g = {S}_g \subseteq R^+_g
\] 
is an integral extension as well. Let $\fQ$ be a prime ideal of $R^+$ not containing $g$ such that $\fQ R^+_g$ lies over $\fq S_g$. Since $\fq$ contracts to $\fp$ in $R$ by our choice of $\fq$, it follows that $\fQ$ also lies over $\fp$ in $R$. Moreover, $z \in IS \cap R$ implies
\[
z \in I{S}_g \subseteq IR^+_g \subseteq IR^+_\fQ.
\]
As $R$ is integrally closed in $\Frac(R)$ and $\Frac(R^+)/\Frac(R)$ is a normal algebraic field extension, the elements of $\Aut(\Frac(R^+)/\Frac(R))$ (which induce automorphisms of $R^+$ over $R$) act transitively on the prime ideals of $R^+$ that lie over $\fp$ \cite[Chapter V, $\mathsection 2.3$, Proposition 6(i)]{Bou89}\footnote{Bourbaki uses the term quasi-Galois extension in \cite[Chapter V, $\mathsection 2.3$, Proposition 6]{Bou89} as a synonym for a normal algebraic field extension. See the footnote on \cite[Pg. 331]{Bou89} as well as \cite[Chapter V, $\mathsection 9$, no. 3, Definition 2]{Bou03}.}. Since $\fQ$ and $\fP$ are prime ideals of $R^+$ that both lie over $\fp$ and $z \in R$ by the hypothesis of (2), we then get
\[
z \in IR^+_\fQ \Rightarrow z \in IR^+_\fP.
\]
However, $\fP$ is an arbitrary prime ideal of $R^+$, and so, $z \in IR^+$, hence also $z \in IR^+ \cap R$, proving $(2)$.

(3) The non-trivial inclusion is $IS^+ \cap R \subseteq IR^+ \cap R$. Suppose $z \in IS^+ \cap R$. Since $S^+$ is a union of module-finite $S$-subalgebras, there exists a module-finite extension $T$ of $S$ contained in $S^+$, such that $z \in IT \cap R$. However, the composition $R \hookrightarrow S \hookrightarrow T$ is quasi-finite since $R \hookrightarrow S$ and $S \hookrightarrow T$ are quasi-finite. Moreover, the induced map $\Spec(T) \rightarrow \Spec(R)$ remains surjective. Then $z \in IT \cap R$ implies $z \in IR^+ \cap R$ by $(2)$.

(4) Note that by definition of a splinter, $R$ is automatically Noetherian. Since purity of $R \rightarrow S$ can be checked locally, we may assume $R$ is local. By Lemma \ref{lem:normal-approx-Gor}, it suffices to show that $R \rightarrow S$ is cyclically pure, that is, for any ideal $I$ of $R$, $IS \cap R = I$. But this follows from (3), since we have
\[
I \subseteq IS \cap R \subseteq IS^+ \cap R = IR^+ \cap R = I.
\]
Here the final equality follows by cyclic purity of $R \rightarrow R^+$ because $R$ is a splinter (Lemma \ref{lem:absolute-int-cl-splinter}). 
\end{proof}

Armed with Proposition \ref{prop:ideal-identity-hensel}, the proof of Theorem \ref{Theorem-B} is now a formal exercise.

\begin{proof}[Proof of Theorem \ref{Theorem-B}]
If $R^h$ is a splinter, then by descent of splinters along the faithfully flat map $R \rightarrow R^h$ (Lemma \ref{prop:splinters-descend-cyc-pur}), it follows that $R$ is a splinter. We now show the converse, that is, we show that if a Noetherian local ring $(R,\fm)$ is a splinter, then $R^h$ is a splinter. Since $R$ is normal (see Lemma \ref{lem:local-splinters-domain}), $R^h$ is normal. To show that $R^h$ is a splinter, it then suffices to show by Lemma \ref{lem:absolute-int-cl-splinter} that the map 
\[
R^h \rightarrow (R^h)^+
\]
is cyclically pure. By normality of $R^h$, purity of the map $R^h \rightarrow (R^h)^+$ will follow by Lemma \ref{lem:normal-approx-Gor} if we can show that for every $\fm R^h$-primary ideal $I$ of $R^h$, the induced map
\[
R^h/I \rightarrow (R^h)^+/I(R^h)^+
\]
is injective. 

Therefore, let $I$ be an $\fm R^h$-primary ideal. By Lemma \ref{lem:m-primary-ideals-Henselization}, there exists a unique $\fm$-primary ideal $J$ of $R$ such that 
\[
I = JR^h.
\] 
We then have 
\begin{equation}
\label{eq:ideal-identity-Henselization}
(I(R^h)^+ \cap R^h) \cap R = J(R^h)^+ \cap R = JR^+ \cap R = J,
\end{equation}
where the first equality is obvious, the second equality follows from Proposition \ref{prop:ideal-identity-hensel}, and the third equality follows from cyclic purity of $R \rightarrow R^+$ because $R$ is a splinter (Lemma \ref{lem:absolute-int-cl-splinter}). Observe that the ideal
\[
I(R^h)^+ \cap R^h
\]
is $\fm R^h$-primary. Indeed, $I(R^h)^+ \cap R^h$ contains the $\fm R^h$-primary ideal $I$, and it is also contained in $\fm R^h$ because $J \neq R$. The upshot of this observation is that $I(R^h)^+ \cap R^h$ must also be expanded from a unique $\fm$-primary ideal $J'$ of $R$. Then uniqueness and (\ref{eq:ideal-identity-Henselization}) forces $J' = J$. Thus, $I(R^h)^+ \cap R^h = JR^h =  I$, which is equivalent to the injectivity of $R^h/I \rightarrow (R^h)^+/I(R^h)^+$. This completes the proof. 
\end{proof}

\begin{remark}
Theorem \ref{Theorem-B} immediately implies a special case of Theorem \ref{Theorem-A}, namely that if $(R,\fm) \hookrightarrow (S,\fn)$ is an \'etale homomorphism of Noetherian local rings such that the induced map on residue fields is an isomorphism, then $R$ is a splinter if and only if $S$ is a splinter. Since $R \rightarrow S$ is faithfully flat, $R$ is a splinter if $S$ is by faithfully flat descent of the splinter property (Proposition \ref{prop:splinters-descend-cyc-pur}). Conversely, suppose $R$ is a splinter. As the Henselization of $(S,\fn)$ is also $R^h$, it follows that there is a faithfully flat local map $S \rightarrow R^h = S^h$. Now $R^h$ is a splinter by Theorem \ref{Theorem-B}, and so, $S$ is a splinter again by faithfully flat descent.
\end{remark}

As a consequence of Theorem \ref{Theorem-B}, we can show that the map 
$R \rightarrow R^+$
is not just pure for a splinter domain $R$, but also \emph{strongly pure} 
in the sense that all induced maps on the local rings of $\Spec(R^+) \rightarrow \Spec(R)$
are pure (see \cite[$\mathsection 2$]{CGM16} for more on strong purity). 
For example, any faithfully flat map is always strongly pure, and since purity is a local
condition, it follows that any strongly pure homomorphism $R \rightarrow S$ that is 
surjective on $\Spec$ is pure. However, there exist pure maps that are not
strongly pure \cite[Corollary 5.6.2]{CGM16} even for maps between 
finite type algebras over a field.

\begin{corollary}
\label{cor:R+-localization}
Suppose $R$ is a domain which is a splinter. Let $\fQ \in \Spec(R^+)$ and let $\fq \coloneqq \fQ \cap R$. Then the induced map on local rings $R_\fq \rightarrow (R^+)_\fQ$ is pure.
\end{corollary}

\begin{proof}
Note that $(R^+)_\fQ$ is a further localization of $(R^+)_\fq$, so the purity of $R_\fq \rightarrow (R^+)_\fQ$ 
is not an immediate consequence of the purity of $R_\fq \rightarrow (R^+)_\fq = (R_\fq)^+$ 
(this latter map is pure because $R_\fq$ is a splinter). 
Replacing $R$ by $R_\fq$ and $R^+$ by $(R^+)_\fq = (R_\fq)^+$,
we may assume that $R$ is local and $\fq$ is the maximal ideal of $R$. 
We have to show that $R \rightarrow (R^+)_\fQ$ is pure. 
Since $R$ is normal, its Henselization $R^h$ can be identified
as a subring of $\Frac(R^+)$. Moreover, an argument of Artin 
(see the proof of \cite[Theorem 2.2]{Art71})
shows that $(R^h)^+$ is the localization of $R^+$ at a prime ideal $\fP$ of $R^+$ that
lies above the maximal ideal of $R$. Then $R \rightarrow R^h \rightarrow (R^h)^+ = (R^+)_\fP$
is pure because it is a composition of two pure maps: 
$R \rightarrow R^h$ is pure by faithful flatness 
and $R^h \rightarrow (R^+)_\fP$ is pure by Theorem \ref{Theorem-B}. 
However, both $\fQ$ and $\fP$ lie
above the maximal ideal of $R$. By 
\cite[Chapter V, $\mathsection 2.3$, Proposition 6(i)]{Bou89}, 
choose $\varphi \in \Aut(R^+/R)$ such that 
$\varphi(\fP) = \fQ$. Then $\varphi$ induces an $R$-algebra isomorphism
$(R^+)_\fP \xrightarrow{\sim} (R^+)_\fQ$, and so, the composition 
$R \rightarrow (R^+)_\fP \xrightarrow{\sim} (R^+)_\fQ$ is pure, as desired.
\end{proof}


As far as we are aware, it is not known whether the Henselization of an \emph{arbitrary} weakly $F$-regular Noetherian local ring of prime characteristic $p > 0$ is weakly $F$-regular. The best known result, due to Hochster and Huneke, is that weak $F$-regularity is preserved under Henselization for local $G$-rings \cite[Theorem (7.24)]{HH94(a)}. Recall that we a say a Noetherian ring $R$, not necessarily local, is a \emph{$G$-ring} if for all maximal ideals $\fm$ of $R$, the formal fibers of $R_\fm$ are geometrically regular (c.f. Remark \ref{rem:G-ring-Henselization}). Note this implies that for any $\fp \in \Spec(R)$, the formal fibers of $R_\fp$ are geometrically regular \cite[Th\'eor\`eme 5.3.1]{ILO14}.  At first glance, \cite[Theorem (7.24)]{HH94(a)} implies the seemingly more general statement that weak $F$-regularity of a Noetherian local ring $(R,\fm)$ is preserved under Henselization if $R^h$ is a $G$-ring and the singular locus of $R^h$ is closed. However, $R^h$ is a $G$-ring if and only if $R$ is a $G$-ring \cite[Theorem 5.3]{Gre76}, and the singular locus of a local $G$-ring is always closed \cite[Expos\'e I, Proposition 5.5.1]{ILO14}. Thus, the requirement for the singular locus of the target to be closed is unnecessary in the statement of \cite[Theorem (7.24)]{HH94(a)}, and the best known result is indeed the one stated above.

Nevertheless, Theorem \ref{Theorem-B} has the following consequence for arbitrary weakly $F$-regular local rings.

\begin{corollary}
\label{Henselization-F-regular}
Let $(R, \fm)$ be a Noetherian local ring of prime characteristic $p > 0$. If $R$ is weakly $F$-regular (i.e. all ideals are tightly closed), then $R^h$ is a splinter.
\end{corollary}

\begin{proof}
Weakly $F$-regular local rings are splinters (Remark \ref{rem:charp-splinters}(2)). Now apply Theorem \ref{Theorem-B}.
\end{proof}


\subsection{Completion of a splinter} \label{subsec:Completion of splinters}
We want to show that if $(R,\fm)$ is a $G$-ring, then $\widehat{R}$ is also a splinter. Before we prove this result, we give an example that illustrates that the  completion of a local splinter may not be always be a splinter, even for rings with very mild singularities. Our example comes from the theory of $F$-singularities summarized in Subsection \ref{subsec:splinters-F-singularities}.

\begin{example}
\label{ex:completion-not-splinter}
Let $(R,\fm)$ be a Noetherian local ring of prime characteristic $p > 0$. It is well-known that if $R$ is Gorenstein, then $R$ is weakly $F$-regular if and only if  $R$ is $F$-rational \cite[Corollary 4.7]{HH94(a)}. Loepp and Rotthaus construct an example of a non-excellent Gorenstein local domain $(R,\fm)$  which is weakly $F$-regular (equivalently $F$-rational), but whose completion is not weakly $F$-regular (equivalently $F$-rational) \cite[Section 5]{LR01}. However, any weakly $F$-regular domain of prime characteristic is a splinter (Remark \ref{rem:charp-splinters}(2)). Thus, Loepp and Rotthaus's construction gives a Gorenstein local splinter domain, whose completion is not $F$-rational. Since $\widehat{R}$ is excellent, if it is a splinter, then it will also be $F$-rational by Remark \ref{rem:charp-splinters}(4), which shows that $\widehat{R}$ cannot be a splinter.
\end{example}

The main technical result used in the proof of Theorem \ref{Theorem-C} is \cite[Proposition 5.10]{Smi94}, which is the analogue for completion of Proposition \ref{prop:ideal-identity-hensel}. We need the following minor generalization of Smith's result.

\begin{proposition}[c.f. {\cite[Proposition 5.10]{Smi94}}]
\label{prop:plus-closure-completion}
Suppose $(R,\fm)$ is a Noetherian normal local ring, and let $\widehat{R}$ denote the completion of $R$. If $R$ is a $G$-ring, then for any ideal $I$ of $R$,
\[
I(\widehat{R})^+ \cap R = IR^+ \cap R = I(R^h)^+ \cap R.
\]
\end{proposition}

\begin{proof}
Recall that we have a chain of faithfully flat maps $R \rightarrow R^h \rightarrow \widehat{R}$, and $\widehat{R}$ is the $\fm R^h$-adic completion of $R^h$. Since $R$ is a $G$-ring, so is its Henselization $R^h$ \cite[Theorem 5.3]{Gre76}. Moreover, a Henselian $G$-ring is excellent \cite[Expos\'e 1, Corollaire 6.3(ii)]{ILO14}, and so, $R^h$ is an excellent, normal local domain. Therefore, $\widehat{R} = \widehat{R^h}$ is also normal.

Let $I$ be an ideal of $R$. The equality 
$$IR^+ \cap R = I(R^h)^+ \cap R $$ follows from Proposition \ref{prop:ideal-identity-hensel}. By the discussion in the previous paragraph, applying \cite[Proposition 5.10]{Smi94} to the map of normal excellent local rings $R^h \rightarrow \widehat{R^h} = \widehat{R}$, we get
\begin{equation}
\label{eq:henselized-plus-closure}
I(R^h)^+ \cap R^h = I(\widehat{R})^+ \cap R^h.
\end{equation}
Intersecting (\ref{eq:henselized-plus-closure}) with $R$ then gives us
\begin{equation*}
IR^+ \cap R = I(R^h)^+ \cap R = (I(R^h)^+ \cap R^h) \cap R \stackrel{(\ref{eq:henselized-plus-closure})}{=} (I(\widehat{R})^+ \cap R^h) \cap R = I(\widehat{R})^+ \cap R,
\end{equation*}
which completes the proof of the proposition.
\end{proof}

\begin{remark}
\label{rem:G-ring-Henselization}
The proof of Proposition \ref{prop:plus-closure-completion} uses that if $(R,\fm)$ is a Noetherian local $G$ ring, then $R^h$ is excellent. The converse is also true. Namely, if $R^h$ is excellent, then $R \rightarrow \widehat{R}$ is regular. To see this, note that $R \rightarrow \widehat{R}$ factors as $R \rightarrow R^h \rightarrow \widehat{R}$, and $\widehat{R}$ is also the completion of $R^h$. Thus $R \rightarrow \widehat{R}$ is regular because $R \rightarrow R^h$ is regular \cite[\href{https://stacks.math.columbia.edu/tag/07QQ}{Tag 07QQ}]{stacks-project}, $R^h \rightarrow \widehat{R}$ is regular by excellence of $R^h$, and the composition of regular maps of Noetherian rings is regular \cite[\href{https://stacks.math.columbia.edu/tag/07QI}{Tag 07QI}]{stacks-project}. 
\end{remark}

We can now prove Theorem \ref{Theorem-C}.

\begin{proof}[Proof of Theorem \ref{Theorem-C}]
Again, by faithfully flat descent of the splinter property, the non-trivial implication is to show that if $(R,\fm)$ is a $G$-ring which is a splinter, then $\widehat{R}$ is also a splinter. 

So suppose $R$ is a splinter. By Lemma \ref{lem:local-splinters-domain}, $R$ is a normal domain. The proof that $\widehat{R}$ is a splinter is now a formal consequence of Proposition \ref{prop:plus-closure-completion} and Lemma \ref{lem:m-primary-ideals-Henselization},  with the argument proceeding in the same manner as in the proof of Theorem \ref{Theorem-B}, upon replacing $R^h$ by $\widehat{R}$ and $(R^h)^+$ by $(\widehat{R})^+$. Thus, the details are omitted.
\end{proof}

Theorem \ref{Theorem-C} allows us to generalize results on excellent local splinters in prime characteristic to splinters that are $G$-rings.

\begin{corollary}
\label{cor:G-splinters-CM-F-rational}
Let $(R,\fm)$ be a Noetherian local $G$-ring of prime characteristic. If $R$ is a splinter, then we have the following:
\begin{enumerate}
	\item $R$ is Cohen--Macaulay.
	\item If $R$ is also Gorenstein, then $R$ is weakly $F$-regular, and hence $F$-rational.
	\end{enumerate} 
\end{corollary}

\begin{proof}
(1) Since $\widehat{R}$ is an excellent local splinter, it is Cohen--Macaulay by Remark \ref{rem:charp-splinters}(3). Thus, $R$ is also Cohen--Macaulay.

(2) Since $\widehat{R}$ is an excellent splinter, it is $F$-rational by Remark \ref{rem:charp-splinters}(4). However, a Gorenstein $F$-rational local ring is weakly $F$-regular \cite[Corollary 4.7]{HH94(a)}. Thus $\widehat{R}$ is weakly $F$-regular, and since weak $F$-regularity descends under faithfully flat maps \cite[Proposition 4.12]{HH90}, it follows that $R$ is weakly $F$-regular. Consequently, $R$ is also $F$-rational.
\end{proof}

\begin{remark}
Since the completion of a regular local ring is always regular, the direct summand theorem shows that the completion of a regular local ring is always a splinter, even though it is not difficult to construct regular local rings whose formal fibers are not geometrically regular. Thus, while $G$-rings are sufficient for the splinter property to ascend under completions, they are by no means necessary.
\end{remark}

As another prime characteristic consequence of Theorem \ref{Theorem-C}, 
we can partially remove finite type assumptions from a result of Ma 
that relates the (derived) splinter condition to a vanishing 
condition on maps of Tor.

\begin{corollary}[c.f. {\cite[Theorem 5.5]{Ma18}}]
\label{splinters-tor-vanishing}
Let $S$ be a (derived) splinter of prime characteristic $p > 0$ that is a $G$-ring
(for example, if $S$ is locally excellent).
Consider a diagram of rings $A \rightarrow R \rightarrow S$ of the same characteristic 
where $A$
is regular and $R$ is a module-finite torsion free $A$-algebra. Then for any $A$-
module $M$, the induced map 
\[
\Tor^A_i(M,R) \rightarrow \Tor^A_i(M,S)
\]
is the zero map for all $i \geq 1$.
\end{corollary}

\begin{proof}
The main point is that Theorem \ref{Theorem-C} allows us to reduce the proof
of this corollary to the case where $A$ is excellent and $A,R,S$ are domains, at which 
stage one can adapt an argument of Hochster and Huneke from the case where $S$ is regular. 
First note that the distinction between derived splinters and splinters is superfluous 
by Bhatt's result \cite[Theorem 1.4]{Bha12}. So we work with the splinter condition. 
Since the completion of any localization of $S$ at a prime ideal preserves the splinter
property by Lemma \ref{lem:splinters-local} and Theorem \ref{Theorem-C}, the proof
of \cite[Theorem 4.1]{HH95} shows that one can reduce to the case where both $A$
and $S$ are complete local (in particular, $A$ is excellent) and $A, R, S$ are domains.
The composite map 
\begin{equation}
\label{eq:tor}
\Tor^A_i(M,R) \rightarrow \Tor^A_i(M,S) \rightarrow \Tor^A_i(M,S^+)
\end{equation}
factors through $\Tor^A_i(M,R^+)$ because the map of domains $R \rightarrow S$
always induces a map $R^+ \rightarrow S^+$ \cite[Proposition (1.2)]{HH95}. We can 
identify $R^+$ with $A^+$ because $R$ is a module-finite extension of $A$.
Since $A$ is an excellent regular ring of prime characteristic $p > 0$, it follows
that $A^+$ is a flat $A$-algebra \cite[(6.7)]{HH92}. 
Thus, $\Tor^A_i(M,R^+) = \Tor^A_i(M,A^+) = 0$
for all $i \geq 1$. This means that (\ref{eq:tor}) is the zero map for all $i \geq 1$.
But $S \rightarrow S^+$ is a pure map because $S$ is a splinter. Thus,
$\Tor^A_i(M,S) \rightarrow \Tor^A_i(M,S^+)$ is injective for all $i$ by 
\cite[Lemma $2.1^\circ$]{HR76}, which forces $\Tor^A_i(M,R) \rightarrow \Tor^A_i(M,S)$
to be the zero map for $i\geq 1$.
\end{proof}

\begin{remark}
\label{weakly-functorial-BCM}
{\*}
\begin{enumerate}
\item 
In \cite[Corollary 4.5]{Ma18}, Ma proves a partial converse of Corollary 
\ref{splinters-tor-vanishing}.
Namely, if $S$ is a homomorphic image of a regular ring\footnote{Ma states 
\cite[Corollary 4.5]{Ma18} for $S$ that is complete local,
or local and essentially of finite type over a field, although his proof only uses 
that $S$ is a homomorphic image of a regular ring.} (for example, if 
$S$ is essentially of finite type over a complete local ring, or $F$-finite by a 
result of Gabber \cite[Remark 13.6]{Gab04}) and satisfies the vanishing
condition on maps of $\Tor$ for all diagrams $A \rightarrow R \rightarrow S$ with 
assumptions as in the statement of Corollary \ref{splinters-tor-vanishing}, then $S$
is a splinter. Moreover, Ma's argument is independent of the characteristic 
of $S$ as long as all rings $A, R, S$ have the same characteristic. However, we do 
not know if the converse of Corollary \ref{splinters-tor-vanishing} holds when $S$ 
is a $G$-ring of prime characteristic that
is not a homomorphic image of a regular ring. Note that even in the class of excellent
local rings there exist ones that are not homomorphic images of regular (or Gorenstein) local rings
by Nishimura's example of an excellent local ring that does not possess a dualizing
complex \cite[Example 6.1]{Nis12}. 

\item A key point in the proof of Corollary \ref {splinters-tor-vanishing} is that excellent local splinters $R$ of prime characteristic admit pure maps to big Cohen--Macaulay $R$-algebras in a weakly functorial fashion. Recent work of Bhatt \cite{Bha20}
implies that a similar result holds in mixed characteristic.
Suppose $(R,\frak m)$ is a Noetherian local splinter which is a $G$-ring and which has residue characteristic $p > 0$. Define
$\mathcal{B}(R)$ to be the $p$-adic completion of $(\widehat{R}^{\frak m})^+$, 
where the latter is the absolute integral closure of the $\fm$-adic 
completion of $R$. The notation $(\widehat{R}^{\frak m})^+$ makes
sense because $\widehat{R}^{\frak m}$ is a splinter by
Theorem \ref{Theorem-C}, and hence, a domain.
The ring $\mathcal{B}(R)$ is a Cohen-Macaulay
$\widehat{R}^{\frak m}$-algebra by \cite[Corollary 5.17]{Bha20},
and so, is also a Cohen--Macaulay $R$-algebra. We now claim that
$R \rightarrow \mathcal{B}(R)$ is pure.
First note that
the composition $\widehat{R}^{\frak m} \rightarrow (\widehat{R}^{\frak m})^+ \rightarrow \mathcal{B}(R)$ is pure. Indeed, 
since $\widehat{R}^{\frak m}$ is normal, it suffices to show
by Lemma \ref {lem:normal-approx-Gor} that for any $\frak{m}\widehat{R}^{\frak m}$-primary ideal $I$, the induced map
\[
\widehat{R}^{\frak m}/I \rightarrow (\widehat{R}^{\frak m})^+/I (\widehat{R}^{\frak m})^+ \rightarrow \mathcal{B}(R)/I \mathcal{B}(R)
\] 
is injective. The first map $\widehat{R}^{\frak m}/I \rightarrow (\widehat{R}^{\frak m})^+/I (\widehat{R}^{\frak m})^+$ is injective because $\widehat{R}^{\frak m}$ is a splinter, and so, $\widehat{R}^{\frak m} \rightarrow (\widehat{R}^{\frak m})^+$ is pure. We claim that
the second map $(\widehat{R}^{\frak m})^+/I (\widehat{R}^{\frak m})^+ \rightarrow \mathcal{B}(R)/I \mathcal{B}(R)$ is an isomorphism. Recall that $\mathcal{B}(R)$ is the 
$p$-adic completion of $(\widehat{R}^{\frak m})^+$, and choose some $p^n \in I$, for $n > 0$. Then
$(\widehat{R}^{\frak m})^+/I (\widehat{R}^{\frak m})^+ \rightarrow \mathcal{B}(R)/I \mathcal{B}(R)$
is a base change of 
\[
(\widehat{R}^{\frak m})^+/p^n(\widehat{R}^{\frak m})^+ \rightarrow \mathcal{B}(R)/p^n \mathcal{B}(R),
\] 
which is an isomorphism by \cite[\href{https://stacks.math.columbia.edu/tag/05GG}{Tag 05GG}]{stacks-project}. Thus, $\widehat{R}^{\frak m} \rightarrow \mathcal{B}(R)$ is pure, and consequently, so is the composition
$R \rightarrow \widehat{R}^{\frak m} \rightarrow \mathcal{B}(R)$.
Finally, for a local homomorphism $(R,\frak m) \rightarrow (S,\frak n)$ of local rings
of residue characteristic $p$ whose respective completions are
domains, we get an induced map
$\mathcal{B}(R) \rightarrow \mathcal{B}(S)$ by functoriality of ideal-adic completions and weak functoriality of 
absolute integral closures.
\end{enumerate}

\end{remark}

\subsection{\'Etale ascent}\label{subsec:\'Etale ascent}
To prove the ascent of the splinter property under essentially \'etale maps (Theorem \ref{Theorem-A}), we will reduce to the finite \'etale case. For the reduction, we need the following result which is well-known to experts. However, a proof is included for the sake of completeness.

\begin{lemma}
\label{lem:Henselization-unramified-maps}
Let $\varphi\colon (R, \fm) \rightarrow (S, \fn)$ be a local homomorphism of local rings (not necessarily Noetherian). Suppose $S$ is the localization of a finite type $R$-algebra $S'$ at a prime ideal $\fq$ such that $R \rightarrow S'$ is quasi-finite at $\fq$. Then the induced map $\varphi^h: R^h \rightarrow S^h$ is finite. 

\noindent In particular, if $\varphi\colon (R, \fm) \rightarrow (S, \fn)$ is an \'etale homomorphism of local rings, then $\varphi^h\colon R^h \rightarrow S^h$ is a finite \'etale map.
\end{lemma}
 
\begin{proof}
Note that $\fq$ lies over $\fm$ because $\fn$ lies over $\fm$. Since the locus of primes of $\Spec(S')$ at which $\Spec(S') \rightarrow \Spec(R)$ is quasi-finite is open \cite[\href{https://stacks.math.columbia.edu/tag/00QA}{Tag 00QA}]{stacks-project}, we may assume $R \rightarrow S'$ is quasi-finite. Then by base change, $R^h \rightarrow R^h \otimes_R S'$ is also quasi-finite \cite[\href{https://stacks.math.columbia.edu/tag/00PP}{Tag 00PP, part (3)}]{stacks-project}. Note that 
\[
(R^h \otimes_R S') \otimes_{S'} \kappa(\fq) = 
 (R^h \otimes_R \kappa(\fm)) \otimes_{\kappa(\fm)} \kappa(\fq) = \kappa(\fq).
\] 
Hence the expansion of $\fq$ in $R^h \otimes_R S = (R^h \otimes_R S') \otimes_{S'} S'_\fq$ is a prime ideal of $R^h \otimes_R S$, which implies by \cite[\href{https://stacks.math.columbia.edu/tag/05WP}{Tag 05WP}]{stacks-project} that 
\[
S^h  = (R^h \otimes_R S)_ {\fq(R^h \otimes_R S)}.
\] 

Let $\fQ$ be the contraction of the maximal ideal of $S^h$ to $R^h \otimes_R S'$.
It is clear that $\fQ$ lies over the maximal ideal $\fm R^h$ of $R^h$. 
Since $R^h$ is Henselian and $R^h \rightarrow R^h \otimes_R S'$ is quasi-finite, \cite[$\mathsection 2.3$, Proposition 4(e)]{BLR90} (see also \cite[\href{https://stacks.math.columbia.edu/tag/04GG}{Tag 04GG}, part (13)]{stacks-project}) shows that we can choose $f \notin \fQ$ such that 
$$(R^h \otimes_R S')_f$$ 
is a finite $R^h$-algebra, which consequently decomposes as a finite product of $R^h$-module finite Henselian local rings \cite[\href{https://stacks.math.columbia.edu/tag/04GG}{Tag 04GG}, part (10) and \href{https://stacks.math.columbia.edu/tag/04GH}{Tag 04GH}]{stacks-project}. Since $\fQ$ lies over the maximal ideal $\fm R^h$ of $R^h$, the  prime ideal $\fQ (R^h \otimes_R S')_f$ is maximal in $(R^h \otimes_R S')_f$ by module finiteness of $R^h \rightarrow (R^h \otimes_R S')_f$. Thus, $S^h = (R^h \otimes_R S)_ {\fq(R^h \otimes_R S)} = (R^h \otimes_R S')_{\fQ}$ coincides with one of the $R^h$-module-finite factors of $(R^h \otimes_R S')_f$, proving the first assertion.

If $\varphi$ is an \'etale homomorphism of local rings, we may assume that $S'$ is an \'etale $R$-algebra. Then $S^h$ is a localization of the \'etale $R^h$-algebra $R^h \otimes_R S'$, and so $\varphi^h$ is an \'etale homomorphism of local rings which is finite by what we proved above because \'etale maps are quasi-finite \cite[\href{https://stacks.math.columbia.edu/tag/00U5}{Tag 00U5}]{stacks-project}.
\end{proof}

\begin{remark}
The analogue of Lemma \ref{lem:Henselization-unramified-maps} also holds when we replace `Henselizations' by `completions' in the Noetherian setting, i.e., if $(R,\fm) \rightarrow (S,\fn)$ is an \'etale (even unramified) homomorphism of Noetherian local rings, then the induced map on completions $\widehat{R}^\fm \rightarrow \widehat{S}^\fn$ is module-finite. This fact is easier to prove \cite[\href{https://stacks.math.columbia.edu/tag/039H}{Tag 039H, part (1)}]{stacks-project}.
\end{remark}

We finally have all the tools in our arsenal to prove Theorem \ref{Theorem-A}.

\begin{proof}[Proof of Theorem \ref{Theorem-A}]
Let $\varphi\colon R \rightarrow S$ be an essentially \'etale homomorphism of Noetherian rings such that $R$ is a splinter. We want to show that $S$ is a splinter. Let $\fq$ be a prime ideal of $S$. By Lemma \ref{lem:splinters-local}, it suffices to show $S_\fq$ is a splinter. Replacing $\varphi\colon R \rightarrow S$ by the induced map $R_{\varphi^{-1}(\fq)} \rightarrow S_\fq$, we may assume $R = (R, \fm)$ and $S = (S, \fn)$ are local, and $\varphi$ is an \'etale local homomorphism of Noetherian local rings. Consider the commutative diagram
\[
\begin{tikzcd}
  R \arrow[r, "\varphi"] \arrow[d]
    & S \arrow[d] \\
  R^h \arrow[r, "\varphi^h"]
& S^h 
\end{tikzcd},
\]
where the vertical maps are faithfully flat (hence pure) and $\varphi^h$ is finite \'etale by Lemma \ref{lem:Henselization-unramified-maps}. Note that $R^h$ is a splinter by Theorem \ref{Theorem-B}, and so, to show that $S$ is a splinter, it is enough to show $S^h$ is a splinter by faithfully flat descent (see Lemma \ref{lem:absolute-int-cl-splinter}). 

Thus, we reduce the proof of Theorem \ref{Theorem-A} to the case where $\varphi \colon (R, \fm) \rightarrow (S, \fn)$ is a finite \'etale local homomorphism of Noetherian local rings. In particular, $\varphi$ induces a surjective map on $\Spec$ since it is faithfully flat. Let $\phi \colon S \rightarrow T$ be a finite map such that the induced map on $\Spec$ is surjective. In order to show that $\phi$ splits, it suffices to exhibit an $S$-linear map $T \rightarrow S$ that maps $1 \in T$ to a unit in $S$.

The composition $R \xrightarrow{\varphi} S \xrightarrow{\phi} T$ is a finite map and $\Spec(\phi \circ \varphi)$ is surjective. Thus, since $R$ is a splinter, there exists an $R$-linear map
\[
g \colon T \rightarrow R
\]
such that $g(1) = 1$. The retraction $g$ induces an $S$-linear map
\[
\Psi_g \colon T \rightarrow \Hom_R(S,R)
\]
defined as follows: for all $t \in T$ and $s \in S$,
\[
\Psi_g(t)(s) \coloneqq g(\phi(s)t).
\]
We then have a commutative diagram
\[
\begin{tikzcd}[column sep=large]
  T \arrow[r, "\Psi_g"] \arrow[dr, "g"]
    & \Hom_R(S,R) \arrow[d, "\ev_1"]\\
&R 
\end{tikzcd}
\]
where the vertical map is evaluation at $1$. Now by Grothendieck duality for a proper smooth map \cite[Chapter VII, Theorem 4.1]{Har66} applied to the finite \'etale map $\varphi$, we have isomorphisms of $S$-modules
\[
\Hom_R(S,R) = \Hom_R(\varphi_*S,R) \cong \varphi_*(\Hom_S(S,\varphi^!R)) = \Hom_S(S,\varphi^!R).
\]
Moreover, since $\varphi$ is \'etale, we know that the functor $\varphi^!$ coincides with the pullback functor $\varphi^*$ because the relative canonical bundle of an \'etale map is trivial (see the definition of $f^!$ for a smooth map given on \cite[Chapter VII, $\mathsection 4$, Pg 388]{Har66}). Thus, we have isomorphisms of $S$-modules
\[
\Hom_R(S,R) \cong \Hom_S(S,\varphi^!R) \cong \Hom_S(S,\varphi^*R) = \Hom_S(S,S) \cong S.
\]
In particular, we then get a commutative diagram
\[
\begin{tikzcd}[column sep=large]
  T \arrow[r, "\widetilde{\Psi_g}"] \arrow[dr, "g"]
    & S \arrow[d, "\widetilde{\ev_1}"]\\
&R 
\end{tikzcd}
\]
where $\widetilde{\Psi_g}$ is $S$-linear and $\widetilde{\ev_1}$ is $R$-linear. We claim that 
\[
\widetilde{\Psi_g}(1) \notin \fn.
\]
Indeed, if $\widetilde{\Psi_g}(1) \in \fn = \fm S$, then by $R$-linearity of $\widetilde{\ev_1}$, we get
\[
1 = g(1) = \widetilde{\ev_1}(\widetilde{\Psi_g}(1)) \in \widetilde{\ev_1}(\fm S) \subseteq \fm,
\]
which is a contradiction. Thus, $\widetilde{\Psi_g}(1) \notin \fn$, which shows that $\widetilde{\Psi_g}$ is an $S$-linear map that sends $1 \in T$ to a unit in $S$. But this is precisely what we wanted to show.
\end{proof}

\begin{corollary}
\label{cor:splinter-strict-Henselization}
Let $(R,\fm)$ be a Noetherian local ring. Then $R$ is a splinter if and only if its strict Henselization $R^{sh}$ is a splinter.
\end{corollary}

\begin{proof}
If $R^{sh}$ is a splinter, so too is $R$ by faithfully flat descent. Recall that $R^{sh}$ is a filtered colimit of pairs $(S,\fn)$, where $(R,\fm) \rightarrow (S,\fn)$ is an \'etale homomorphism of Noetherian local rings and $\kappa(\fn)$ is contained in a fixed choice of a separable algebraic closure of $\kappa(\fm)$. Let $R^{sh} \rightarrow T$ be a finite map such that the induced map $\Spec(T) \rightarrow \Spec(R^{sh})$ is surjective. Since $T$ is a finitely presented $R^{sh}$-algebra, there exists a model $T_S$ of $T$ over one of the pairs $(S,\fn)$ and a fibered square
\[
\begin{tikzcd}
  \Spec(T) \arrow[r] \arrow[d, twoheadrightarrow]
    & \Spec(T_S) \arrow[d] \\
  \Spec(R^{sh}) \arrow[r, twoheadrightarrow]
& \Spec(S). \end{tikzcd}
\]
Note that the map $(S,\fn) \rightarrow (R^{sh},\fm)$ is faithfully flat since $R^{sh}$ is also the strict Henselization of $S$. Thus, by faithfully flat descent of finite generation \cite[Chapter 1, $\mathsection 3.6$, Proposition 11]{Bou89}, we have that $T_S$ is a finitely generated $S$-module. Moreover, by commutativity of the above diagram, $\Spec(T_S) \rightarrow \Spec(S)$ is surjective. Since $R$ is a splinter and $S$ is an essentially \'etale extension of $R$, Theorem \ref{Theorem-A} implies $S$ is also a splinter. Therefore $S \rightarrow T_S$ splits, and so by base change, $R^{sh} \rightarrow T$ splits.
\end{proof}

\begin{remark}
\label{rem:main-theorem-A-B-equivalent}
Theorem \ref{Theorem-A} $\Rightarrow$ Corollary \ref{cor:splinter-strict-Henselization}, while faithfully flat descent of the splinter property shows that Corollary \ref{cor:splinter-strict-Henselization} $\Rightarrow$ Theorem \ref{Theorem-B} using the faithfully flat map $R^h \rightarrow R^{sh}$. Thus, Theorem \ref{Theorem-A} implies Theorem \ref{Theorem-B}. Since Theorem \ref{Theorem-A} is deduced as a consequence of Theorem \ref{Theorem-B}, this shows that Theorems \ref{Theorem-A} and \ref{Theorem-B} are equivalent.
\end{remark}

The derived version of Theorem \ref{Theorem-A} also holds with some caveats.

\begin{corollary}
\label{cor:derived-Theorem-A}
Let $\varphi: (R,\fm) \rightarrow (S,\fn)$ be an \'etale homomorphism of Noetherian local rings such that $R$ is a derived splinter.
\begin{enumerate}
	\item If $R$ has prime characteristic $p > 0$ or has mixed characteristic $(0,p)$ where $p$ is in the Jacobson radical of $R$, then $S$ is a derived splinter.
	\item If $R$ is essentially of finite type over a field $k$ of characteristic $0$, then $S$ is a derived splinter. 
\end{enumerate}
\end{corollary}

\begin{proof}
Since $\varphi$ is an \'etale homomorphism of local rings, $R$ and $S$ both have the same characteristic.

$(1)$ follows in prime characteristic $p > 0$ by \cite[Theorem 1.4]{Bha12} and Theorem \ref{Theorem-A} since splinters and derived splinters coincide in prime characteristic. The same is true in mixed characteristic by forthcoming work of Bhatt and Lurie announced in \cite[Remark 1.7]{Bha20}.

$(2)$ Since $R$ being a derived a splinter is equivalent to it having rational singularities (\cite{Kov00} and \cite[Theorem 2.12]{Bha12}), it suffices for us to show that rational singularities ascend under an \'etale homomorphism of local rings. But Elkik showed, more generally, that rational singularities ascend under essentially smooth homomorphisms of Noetherian local rings \cite[Th\'eor\`eme 5]{Elk78}.
\end{proof}

\begin{remark}
Recently Kov\'acs has generalized the notion of a rational singularity to arbitrary characteristic for excellent normal Cohen--Macaulay schemes that admit a dualizing complex \cite[Definition 1.3]{Kov18}. He shows that if $X$ is such a scheme that is also a derived splinter, then $X$ has rational singularities in this more general sense \cite[Theorem 8.7]{Kov18}. In fact, $X$ is also pseudorational in the sense of Lipman and Teissier \cite{LT81} by \cite[Corollary 9.14]{Kov18}. However, the converse is false, that is, a rational singularity is not always a derived splinter. For example, in prime characteristic there exist finite type graded rings over fields with rational singularities that are not $F$-rational \cite[Example (2.11)]{HW96}. Any such ring cannot be a splinter by Remark \ref{rem:charp-splinters}(4), hence also not a derived splinter by \cite[Theorem 1.4]{Bha12}. 
\end{remark}

\section{Some open questions}
\label{sec:open-questions}

We conclude this paper with some questions that we believe are open for splinters. The first, is the generalization of Theorem \ref{Theorem-A} mentioned in the introduction:

\noindent \textbf{Question 1:} Suppose $\varphi: R \rightarrow S$ is a flat homomorphism of Noetherian rings with geometrically regular fibers. If $R$ is a splinter, is $S$ a splinter?

\begin{remark}
\label{rem:question-1-implies-dsc}
{\*}
\begin{enumerate}
\item Question 1 reduces to the following special case -- 
\begin{center}
\textbf{Question 1$^{\prime}$:} If $R$ is a splinter, is $R[x]$ a splinter? 
\end{center}	
	Indeed, if $R \rightarrow S$ is a regular map, then N\'eron-Popescu desingularization implies that $S$ can be written as a filtered colimit of smooth $R$-algebras \cite{Pop90, Swa98}. Note, however, that the smooth $R$-algebras in the colimit are not necessarily $R$-subalgebras of $S$. If $S \rightarrow T$ is a finite ring map that is surjective on $\Spec$, then by descent of properties of morphisms over projective limits of schemes  \cite[\href{https://stacks.math.columbia.edu/tag/01ZM}{Tags 01ZM}, \href{https://stacks.math.columbia.edu/tag/01ZO}{01ZO} and \href{https://stacks.math.columbia.edu/tag/07RR}{07RR}]{stacks-project}, there exists a model $T_{S'}$ of $T$ over a smooth $R$-algebra $S'$ such that $S' \rightarrow T_{S'}$ is finite and induces a surjection on $\Spec$. To get a splitting of $S \rightarrow T$ it suffices to show that $S' \rightarrow T_{S'}$ splits. Therefore to answer Question 1, we may assume $R \rightarrow S$ is a smooth ring map. However, given $\fq \in \Spec(S)$, in a suitable affine open neighborhood $U$ of $\fq$, the map $U \rightarrow \Spec(R)$ factors as an \'etale map $U \rightarrow \mathbb{A}^n_R$ followed by the canonical projection $\mathbb{A}^n_R \rightarrow \Spec(R)$ \cite[Chapter 2, Remark 12]{BLR90}. Hence by Theorem \ref{Theorem-A} and induction on $n$, it further suffices to show that a polynomial ring over a splinter remains a splinter.

	\item Question $1^{\prime}$ implies the direct summand theorem. In fact, the direct summand theorem is equivalent to Question $1^{\prime}$
when $R$ is excellent and regular. Indeed, by a reduction due to Hochster \cite[Theorem (6.1)]{Hoc83}, it suffices to prove that complete, unramified regular local rings are splinters. In mixed characteristic, any such regular local ring is of the form $V[[x_1,\dots,x_n]]$, where $V$ is a $p$-adically complete (hence also excellent) DVR of mixed characteristic $(0,p)$. 

Consider the factorization
\[
V \rightarrow V[x_1,\dots,x_n]_{(p,x_1,\dots,x_n)} \xrightarrow{(p,x_1,\dots,x_n)-\textrm{adic completion}} V[[x_1,\dots,x_n]].
\]
If we can show that $V[x_1,\dots,x_n]$ is a splinter (which follows if Question $1^{\prime}$ is true), then the localization $V[x_1,\dots,x_n]_{(p,x_1,\dots,x_n)}$ is also a splinter (Lemma \ref{lem:splinters-local}). But  $V[x_1,\dots,x_n]_{(p,x_1,\dots,x_n)}$ is excellent because $V$ is complete, and so, Theorem \ref{Theorem-C} will then imply that its $(p,x_1,\dots,x_n)$-adic completion $V[[x_1,\dots,x_n]]$ is also a splinter. 

	\item Suppose $R$ is a Gorenstein ring of prime characteristic $p > 0$ that is also a $G$-ring. If $R$ is a splinter, then so is $R[x]$. First note that $R[x]$ is also a $G$-ring since the property of being a $G$-ring is preserved under essentially finite type maps \cite[Th\'eor\`eme (7.4.4)]{EGAIV_II}. Let $\fM$ be a maximal ideal of $R[x]$ and $\fp$ denote its contraction to $R$. Since the splinter condition can be checked locally at the closed points (Lemma \ref{lem:local-splinters-domain}), it suffices to show $R[x]_\fM$ is a splinter. Note $R_\fp$ is weakly $F$-regular by Corollary \ref{cor:G-splinters-CM-F-rational} since $R_\fp$ is a splinter and a $G$-ring. Moreover, the local homomorphism $$R_\fp \rightarrow R[x]_\fM$$ is regular, $R[x]_\fM$ is a $G$-ring, and the singular locus of $R[x]_\fM$ is closed by \cite[Expos\'e 1, Proposition 5.5.1(i)]{ILO14}. Thus, $R[x]_\fM$ is also weakly $F$-regular by \cite[Theorem (7.24)]{HH94(a)}. But a weakly $F$-regular ring is always a splinter by Remark \ref{rem:charp-splinters}(2), hence $R[x]_\fM$ is a splinter as desired. 
		
	If $R$ is a locally excellent $\mathbb{Q}$-Gorenstein splinter of prime characteristic $p > 0$, then $R[x]$ is also a splinter. Indeed, $R$ is then \emph{$F$-regular} by \cite[Theorem 1.1]{Sin99} (i.e. all localizations of $R$ are weakly $F$-regular) and a $G$-ring since it is locally excellent. Therefore, by the same reasoning as in the previous paragraph, $R[x]$ is a splinter. A similar argument also applies to $F$-finite rings with finitely generated anti-canonical algebras by \cite{CEMS18}.

\end{enumerate}
\end{remark}

\noindent \textbf{Question 2:} Let $R$ be an excellent ring. Is the locus of primes $\fp \in \Spec(R)$ such that $R_{\fp}$ is a splinter open?

If $R$ is {locally excellent}, but not excellent, then Question 2 has a negative answer by the following general result of Hochster:

\begin{theorem}\cite[Proposition 2]{Hoc73(b)}
\label{thm:non-openness-loci}
Let $\cP$ be a property of Noetherian local rings. Let $k$ be an algebraically closed field, and let $(R,\fm)$ be essentially of finite type over $k$ such that
\begin{enumerate}
	\item $R$ is a domain (hence geometrically integral),
	\item $R/\fm = k$, and
	\item for every field extension $L/k$, the ring $(R \otimes_k L)_{\fm(R \otimes_k L)}$ fails to satisfy $\cP$.
\end{enumerate}
Additionally, suppose any field satisfies $\cP$. For all $n \in \mathbb{N}$, let $R_n$ be a copy of $R$ with maximal ideal $\fm_n = \fm$. Let $R' \coloneqq \bigotimes_{n \in \mathbb{N}} R_n$, where the infinite tensor product is taken over $k$. Then each $\fm_nR'$ is a prime ideal of $R'$. Moreover, if $S = R' \setminus (\bigcup_m \fm_nR')$, then 
\[
T \coloneqq S^{-1}R'
\]
is a Noetherian domain whose locus of primes that satisfy $\cP$ is not open in $\Spec(T)$. Furthermore, each local ring of $T$ is essentially of finite type over appropriate field extensions of $k$, 
hence $T$ is locally excellent.
\end{theorem}

\begin{example}
\label{ex:splinter-locus-not-open}
Let $\cP$ be the property of being a splinter. Note that all fields are splinters. Let $(R,\fm)$ be the local ring of a closed point of an algebraic variety over $k = \overline{k}$ such that $R$ is not a splinter (for example, choose a non-normal $R$). Then for any field extension $L/k$, $(R \otimes_k L)_\fm$  cannot be a splinter because $(R \otimes_k L)_{\fm(R \otimes_k L)}$ is a faithfully flat extension of $R$, and splinters descend under faithfully flat maps (Proposition \ref{prop:splinters-descend-cyc-pur}). Thus, $R$ satisfies requirements $(1)$-$(3)$ of Theorem \ref{thm:non-openness-loci}, and so, the theorem gives us a locally-excellent ring whose splinter locus is not open.
\end{example}

\section{Acknowledgments}
We thank Takumi Murayama, Shravan Patankar, Thomas Polstra, Karl Schwede and Anurag Singh for helpful conversations, for comments on a draft and for their interest. We thank Shravan for pointing out an incorrect citation. We are grateful to Bhargav Bhatt, Mel Hochster and Karen Smith for patiently answering our many questions. We first became aware of some of the questions addressed in this paper during a lecture series taught by Linquan Ma at the University of Illinois at Chicago. Additionally we thank Linquan for helpful discussions pertaining to Remark \ref {weakly-functorial-BCM}. Part of this research was carried out at the birthday conference in honor of Bernd Ulrich at the University of Notre Dame, and we thank the conference organizers for providing a stimulating environment. Finally, we thank the referees for their  suggestions that have improved this paper.

\bibliographystyle{amsalpha}

\end{document}